\documentclass[12pt]{amsart}
\usepackage{amsmath,amsthm,amscd,amssymb,amsfonts}
\usepackage{mathrsfs,upgreek}
\usepackage{hyperref}
\usepackage{mathtools}
\usepackage{graphicx}
\usepackage{fullpage}
\usepackage{cleveref}
\usepackage{tikz-cd}

\makeatletter
\let\@wraptoccontribs\wraptoccontribs
\makeatother

\usepackage[sans]{dsfont} 
\usepackage[T1]{fontenc}
\DeclareMathAlphabet{\mathpzc}{OT1}{pzc}{m}{it} 

\numberwithin{equation}{section} 
\numberwithin{figure}{section} 

\theoremstyle{plain}
\newtheorem{theo}{Theorem}
\newtheorem{prop}{Proposition}[section]
\newtheorem{coro}[prop]{Corollary}
\newtheorem{lemm}[prop]{Lemma}

\theoremstyle{definition}
\newtheorem{defi}[prop]{Definition}

\theoremstyle{remark}
\newtheorem{rema}[prop]{Remark}
\newtheorem{exam}[prop]{Example}

\newtheoremstyle{citing}
  {3pt}
  {3pt}
  {\itshape}
  {}
  {\bfseries}
  {.}
  {.5em}
  {\thmnote{#3}}

\theoremstyle{citing}
\newtheorem*{generic}{}

\newcommand{\N}{\mathbb{N}}

\newcommand{\R}{\mathbb{R}}

\newcommand{\cP}{\mathcal{P}}

\newcommand{\cT}{\mathcal{T}}

\newcommand{\fD}{\mathfrak{D}}

\newcommand{\sB}{\mathscr{B}}
\newcommand{\sC}{\mathscr{C}}
\newcommand{\sD}{\mathscr{D}}

\newcommand{\sF}{\mathscr{F}}

\newcommand{\sL}{\mathscr{L}}
\newcommand{\sM}{\mathscr{M}}

\newcommand{\sP}{\mathscr{P}}
\newcommand{\sQ}{\mathscr{Q}}

\newcommand{\sS}{\mathscr{S}}

\newcommand{\hI}{\widehat{I}}

\newcommand{\hR}{\widehat{R}}

\newcommand{\hvarphi}{\widehat{\varphi}}

\newcommand{\tJ}{\widetilde{J}}

\newcommand{\tP}{\widetilde{P}}

\newcommand{\tY}{\widetilde{Y}}

\newcommand{\teta}{\widetilde{\teta}}

\newcommand{\tpi}{\widetilde{\pi}}

\newcommand{\tphi}{\widetilde{\phi}}
\newcommand{\tvarphi}{\widetilde{\varphi}}

\newcommand{\partn}[1]{{\smallskip \noindent \textbf{#1.}}}
\renewcommand{\:}{\colon}
\renewcommand{\=}{\coloneqq}
\newcommand{\dd}{\hspace{1pt}\operatorname{d}\hspace{-1pt}}

\DeclareMathOperator{\diam}{diam}

\DeclareMathOperator{\supp}{supp} 

\newcommand{\whf}{\widehat{f}}

\newcommand{\uell}{\underline{\ell}}
\newcommand{\hf}{\widehat{f}}
\newcommand{\tf}{\widetilde{f}}
\newcommand{\uz}{\underline{z}}
\newcommand{\mpf}{f}
\newcommand{\tmpf}{\tf}
\newcommand{\hmpf}{\hf}
\newcommand{\Pf}{P}
\newcommand{\sPf}{\sP}

\newcommand{\one}{\mathds{1}}

\newcommand{\Sspace}{C_\dag^{1,\gamma}(\R)}

\begin{document}

\title{Phase transitions in temperature for\\ intermittent maps}
\author{Daniel Coronel}
\address{Facultad de Matem\'aticas, Pontificia Universidad Cat\'olica de Chile, Campus San Joaqu\'in, Avenida Vicu\~{n}a Mackenna 4860, 
Santiago, Chile.} 
\email{acoronel@uc.cl}
\author{Irene Inoquio-Renteria}
\address{Departamento de Matemáticas, Universidad de La Serena, Avenida Juan Cisternas 1200, La Serena, Chile.} 
\email{irene.inoquio@userena.cl}
\author{Juan Rivera-Letelier}
\address{Department of Mathematics, University of Rochester. Hylan Building, Rochester, NY~14627, U.S.A.}
\email{riveraletelier@gmail.com}
\urladdr{\url{http://rivera-letelier.org/}}
\begin{abstract}
This article characterizes phase transitions in temperature \textcolor{black}{for intermittent maps} within a specific space of \textsc{Hölder} continuous potentials, distinguished by their regularity and asymptotic behavior at zero. We also characterize the phase transitions in temperature that are robust within this space. Our results reveal a connection between phase transitions in temperature and ergodic optimization.\end{abstract}

\maketitle

%
%

\section{Introduction}
\label{s:intro}

 Since the last century, the probabilistic approach to studying "complex" dynamical systems has become a significant paradigm. To carry out such a study, one requires a probability measure that is invariant under the dynamics and describes the behavior of most orbits. Thermodynamic formalism has proven to be a powerful tool for selecting such measures in the context of uniformly hyperbolic and expanding maps. For these systems, every sufficiently regular potential~$\varphi$ admits a unique equilibrium state that is fully supported, has positive entropy, and enjoys strong statistical properties. 
Moreover, the pressure function associated with the one-parameter family of potentials $\beta\varphi$, for $\beta \in (0,+\infty)$, is real-analytic in the parameter~$\beta$. This property is commonly referred to as the absence of phase transitions in temperature for the potential.
However, these results generally hold only when the dynamical system is uniformly hyperbolic or expanding, or when the potential is sufficiently regular. Notable examples of phase transitions in temperature include the potentials constructed by~\mbox{\textsc{Hofbauer}} in~\cite{Hof77} for the one-sided full shift on two symbols, and the geometric potential for the \textsc{Manneville–Pomeau} maps~\cite{PreSla92}.

The \textsc{Manneville–Pomeau} family of dynamical systems belongs to the class of intermittent maps, which have been extensively studied in smooth ergodic theory. These maps provide classical examples of non-uniformly expanding dynamics; see, for instance, \cite{BomCar23,BalTod16,BruTerTod19,ChaColRedVer09,CliTho13,CasVar13,GarIno22,GarIno24,Gou04b,Hol05,Hu04,InoRiv26,Klo19b,Klo20,Kor16,LiRiv14a,LivSauVai99,MelTer12,Pia80,PomMan80,PreSla92,PolWei99,Sar01a,Sar02,ShevSt13,Tha00,You99b}.

Most previous work has focused on the geometric potential. However, some papers deal with \textsc{Hölder} continuous potentials; see, for example, \cite{InoRiv26,Klo20,LiRiv14a}. In this article, we take an intermediate approach. Our goal is to characterize phase transitions in temperature for the \textsc{Manneville–Pomeau} family within a specific class of \textsc{Hölder} continuous potentials, distinguished by its regularity and asymptotic behavior at zero (see Theorem~\ref{t:sarig} and Corollary~\ref{c:sarig}).
This class, introduced by \textsc{Sarig} in~\cite{Sar01a}, enables, among other things, the study of the thermodynamic formalism for potentials near the geometric potential. Another advantage of this class is that it allows for a more tractable technical analysis compared to the full class of \textsc{Hölder} continuous potentials. More precisely, the additional regularity of potentials in this class implies a form of “bounded variation of ergodic sums.” At the same time, the asymptotic behavior at zero allows for a clear description of potentials exhibiting robust phase transitions in temperature (see Theorem~\ref{t:robust sarig}).
Moreover, when the indifferent fixed point is “flatter” (i.e., for $\alpha > 1$), this class of potentials reveals an additional regularity in phase transitions in temperature that is not present in the \textsc{Hölder} class (see Theorem~\ref{t:sarig alpha} and Corollary~\ref{c:robust vs nonrobust}). This final result aligns with the philosophy proposed by \textsc{Israel} in~\cite{Isr79}, which states that, to observe interesting phase diagrams, one must consider smaller interaction spaces. In our setting, interactions correspond to potentials.

A disadvantage of our smaller class of potentials is that it excludes the historically significant class of \textsc{Hölder} continuous potentials. We address these in the companion paper~\cite{CorRiv25b}, providing a complete topological description of their phase diagram.

To state our results more precisely, we begin with some definitions.

\subsection{Phase transitions in temperature}
\label{ss:pt}

Let~$\alpha > 0$ be given, and let~$f \colon [0,1] \to [0,1]$ be defined by
\begin{equation}
\label{d:map}
f(x) \= 
\begin{cases}
x(1 + x^\alpha), & \text{if } x(1 + x^\alpha) \le 1;\\
x(1 + x^\alpha) - 1, & \text{otherwise}.
\end{cases}
\end{equation}
Denote by~$\sM$ the space of \textsc{Borel} probability measures that are invariant under~$f$. For every measure~$\mu \in \sM$, we denote by~$h_\mu$ its entropy.
Let~$\varphi$ be a continuous function on~$[0,1]$. We call~$\varphi$ a \emph{potential} and define the \emph{pressure}~$P(\varphi)$ of the potential~$\varphi$ by
\begin{equation}
P(\varphi) \= \sup\left\{ h_\mu + \int \varphi \dd\mu \: \mu \in \sM \right\}.
\end{equation}
A measure~$\mu \in \sM$ that realizes the supremum above is called an \emph{equilibrium state} of~$f$ for~$\varphi$.
We say that a measure~$\nu \in \sM$ is \emph{maximizing for the potential~$\varphi$} if
\begin{equation}
\int \varphi \dd\nu = \sup_{\mu \in \sM} \int \varphi \dd\mu.
\end{equation}

We say that a continuous potential~$\varphi$ undergoes a \emph{phase transition in temperature} if there exists~$\beta_* \in (0,+\infty)$ such that the function $\beta \mapsto P(\beta\varphi)$ is not real-analytic at~$\beta_*$. The terminology comes from Statistical Mechanics, where~$\beta$ is interpreted as the inverse temperature.
If the potential~$\varphi$ is \textsc{Hölder} continuous, then there is at most one point in~$(0,+\infty)$ where the function 
$\beta \mapsto P(\beta\varphi)$ fails to be real-analytic, and if a phase transition occurs at~$\beta_*$, then for all~$\beta \ge \beta_*$, one has $P(\beta\varphi) = \beta \varphi(0)$; see Proposition~\ref{c:tpt} in~\S\ref{ss:dichotomy}.
In particular, if a \textsc{Hölder} continuous potential~$\varphi$ undergoes a phase transition in temperature at~$\beta_* \in (0,+\infty)$, then for every~$\beta \ge \beta_*$, the invariant measure~$\delta_0$ is an equilibrium state for~$\beta\varphi$. Consequently, it is also a maximizing measure for~$\varphi$.
Moreover,~$\delta_0$ is the unique maximizing measure for~$\varphi$ (see~\textcolor{black}{Corollary \ref{c:ergdodic optimization}}).

In Theorem~\ref{t:sarig} and Corollary~\ref{c:sarig} below, we show that for~$\varphi$ in a suitable class of potentials, the uniqueness of~$\delta_0$ as the maximizing measure for~$\varphi$ implies the existence of a phase transition in temperature for~$\varphi$. These results reveal a strong connection between phase transitions in temperature and the theory of Ergodic Optimization~\cite{Jen19}.
The class of potentials we consider behaves asymptotically like~$c x^\gamma$ near zero, with~$c \in \R$ and~$\gamma \in (0,+\infty)$. For example, the geometric potential~$-\log Df$ has this asymptotic form with~$c = -(\alpha + 1)$ and~$\gamma = \alpha$.
Theorem~\ref{t:sarig} shows that the occurrence of a phase transition in temperature is not merely a local property depending on the asymptotic behavior of the potential at zero; it also has a global component that can be characterized via maximizing measures. The dependence of the phase transition on the asymptotic behavior at zero is subtle, as further illustrated by Theorem~\ref{t:sarig}.
However, when the exponent of the system~$\alpha$ and the exponent of the potential~$\gamma$ coincide, the potentials exhibiting a phase transition in temperature display an additional rigidity, as shown in Theorem~\ref{t:sarig alpha}.

\begin{theo}
\label{t:sarig}
Let~$\alpha$ be in~$(0,+\infty)$ and let~$f$ be the \textsc{Manneville–Pomeau} map of parameter~$\alpha$.
Let~$\varphi \: [0,1]\to \R$ be a continuous potential  with continuous derivative on~$(0,1]$ such that
there are~$c$ in $\R$ and~$\gamma>0$ verifying 
\begin{equation}
\label{e:leading}
c=\lim_{x\to 0^+} \frac{\varphi'(x)}{\gamma x^{\gamma-1}}.
\end{equation}
The following hold.
\begin{enumerate}
\item[1.] If $\gamma\le \alpha$ and $c>0$, then $\delta_0$ is not a maximizing measure for~$\varphi$;
\item[2.] If $\gamma\le \alpha$, $c< 0$ and~$\delta_0$ is the unique maximizing measure for~$\varphi$, then~$\varphi$ undergoes a phase transition in temperature; 
\item [3.] If~$\varphi$ undergoes a phase transition in temperature, then~$\gamma \le \alpha, c\le 0$ and~$\delta_0$ is the unique maximizing measure for~$\varphi$.
\end{enumerate}
\end{theo}
 The following result is a stronger version of Theorem~\ref{t:sarig}(3) with $\gamma=\alpha$.
 \begin{theo}
\label{t:sarig alpha}
Let~$\alpha, f, \gamma$ and~$c$ be as in Theorem~\ref{t:sarig}.
If~$\varphi$ undergoes a phase transition in temperature and $\gamma=\alpha$, then~$c<0$.
\end{theo}

 The following corollary is a direct consequence of Theorems~\ref{t:sarig} and~\ref{t:sarig alpha}.
 \begin{coro}
 \label{c:sarig}
 Let~$\alpha, f, \gamma$ and~$c$ be as in Theorem~\ref{t:sarig}. The following hold.
 \begin{enumerate}
 \item[1.] If~$\gamma\le \alpha$ and~$c\neq 0$, then~$\varphi$ undergoes a phase transition in temperature if and only if~$\delta_0$ is the unique maximizing measure for~$\varphi$.
 \item[2.] If~$\gamma = \alpha$, then~$\varphi$ undergoes a phase transition in temperature if and only if~$c<0$ and~$\delta_0$ is the unique maximizing measure for~$\varphi$.
 \end{enumerate}
 \end{coro}
  
Notice that every $\varphi$ as in Theorem~\ref{t:sarig} is \textsc{Hölder} continuous with exponent $\min\{1,\gamma\}$. Also notice that under the hypotheses of Theorem~\ref{t:sarig}, if $\gamma \le \alpha, c<0$ and  $\varphi < \varphi(0)$ on $(0,1]$, then $\varphi$ undergoes a phase transition in temperature. In this case, we give a simpler proof that $\varphi$ undergoes a phase transition in temperature; see Corollary~\ref{c:phase transitions}(2). An example is the geometric potential (Proposition~\ref{ex:geometric}).

The conclusion of Theorem~\ref{t:sarig}(1) is false for $\gamma>\alpha$ as the following example shows.  Put $\delta\= \gamma-\alpha$ and for every $a$ in $(0,+\infty)$ define the potential $\varphi_a(x) \= (\delta/2) x^\gamma (1-ax)$ for every $x$ in $[0,1]$. Now, consider the potential $h(x)\=-x^\delta$ for every $x$ in $[0,1]$,  and the coboundary~$h\circ f - h$. We have that $h\circ f (0) - h(0) =0$, and for every $x$ in $[0,1]$ close to~$0$,  we have $h\circ f (x) - h(x)$ is close to 
$-\delta x^\gamma$. Therefore, by taking $a$ sufficiently large, we can ensure that $\varphi_a + h\circ f - h$ is strictly negative on 
$(0,1]$, which implies that $\delta_0$ is the unique maximizing measure for $\varphi_a + h\circ f - h$ and hence also for 
$\varphi_a$.

The conclusion of Theorem~\ref{t:sarig}(2) does not hold for~$c=0$. Indeed, for~$\gamma'$ in~$(\gamma,+\infty)$ the potential~$\omega_{\gamma'}$ 
defined by~$\omega_{\gamma'}(x)\= -x^{\gamma'}$ satisfies 
\begin{equation}
\label{e:c=0}
\lim_{x\to 0^+} \frac{\omega_{\gamma'}'(x)}{\gamma x^{\gamma-1}} = 0,
\end{equation}
and it undergoes a phase transition in temperature if and only if $\gamma'\le \alpha$, see Corollary~\ref{c:omega-phase-transition}. Thus, if $\gamma<\alpha<\gamma'$ then 
$\omega_{\gamma'}$ does not undergo a phase transition in temperature, and if $\gamma<\gamma'<\alpha$ then 
$\omega_{\gamma'}$ undergoes a phase transition in temperature.
 In both cases,~$\gamma<\alpha$ and~$\delta_0$ is the unique maximizing measure for 
$\omega_{\gamma'}$. The same example shows that the conclusion of Theorem~\ref{t:sarig alpha} does not hold for $\gamma \neq \alpha$. Since we can take 
$\gamma<\gamma'<\alpha$ and thus,
$\omega_{\gamma'}$ undergoes a phase transition in temperature but $c=0$ by \eqref{e:c=0}.

In~\cite[Proposition~1(2)]{Sar01a}, it was stated that, under the hypotheses of Theorem~\ref{t:sarig}, $\gamma\le \alpha$ and $c< 0$, implies that~$\varphi$ undergoes a phase transition in temperature.\footnote{This result is stated for a variant of pressure. For a proof that this variant coincides with the variational pressure considered here, see the last paragraph of Section~\ref{ss:notes}.} However, the following example shows we can not omit the hypothesis that~$\delta_0$ is the unique maximizing measure for~$\varphi$ in Theorem~\ref{t:sarig}(2).
\begin{exam}
Let $\gamma$ be in $(0,+\infty)$ and consider the potential 
\begin{equation}
\hvarphi(x) \= -x^\gamma(1 - x).
\end{equation} 
It satisfies~\eqref{e:leading} with~$c = -1$. We have that~$\delta_0$ and~$\delta_1$ are maximizing measures for~$\hvarphi$. Thus, by Theorem~\ref{t:sarig}(3), the potential~$\hvarphi$ does not undergo a phase transition in temperature.

We now provide an alternative proof that does not rely on Theorem~\ref{t:sarig}. Let~$x_2$ be the smaller preimage of~$1$ under~$f^2$, and let~$X$ be the maximal invariant set of~$f$ in~$[x_2,1]$. Then~$f|_X$ is topologically conjugate to a topologically mixing subshift of finite type, namely the golden mean shift, via a conjugacy that is \textsc{Hölder} continuous. By the thermodynamic formalism for \textsc{Hölder} continuous potentials, it follows that for every~$\beta > 0$, the potential~$\beta \hvarphi|_X$ has a unique equilibrium state with positive entropy; see, for instance, \cite[Theorem~1.25]{Bow08}. Therefore,
\begin{equation}
P(\beta \hvarphi) \ge P_{f|_X}(\beta \hvarphi|_X)
> h_{\delta_1} + \int \beta \hvarphi|_X \dd\delta_1
= \beta \hvarphi(1)
= \beta \hvarphi(0).
\end{equation}
Since this holds for every $\beta>0$, Proposition~\ref{c:tpt}(1) implies that the potential~$\hvarphi$ does not undergo a phase transition in temperature.
\end{exam}

\subsection{Robust phase transitions in temperature}
\label{ss:Robust phase transitions in temperature}

The following results aim to understand when a phase transition in temperature is robust. For this, we introduce a function space containing the potentials introduced in Theorem~\ref{t:sarig}. Denote by $C(\R)$ the space of real-valued continuous functions on $[0,1]$ endowed with the uniform norm $\|{\cdot}\|$ and by 
$C_\dag^1(\R)$ the subspace of $C(\R)$ of functions with continuous derivative on $(0,1]$. 
For every $\gamma$ in $(0,+\infty)$ we define the space of functions $\Sspace$ by
\begin{equation}
\label{d:Sspace}
\Sspace\= \left\{\varphi \in  C_\dag^1(\R) \colon \lim_{x\to 0^+} \frac{\varphi'(x)}{x^{\gamma-1}} \in \R \right\}.
\end{equation}
For every $\varphi$ in $\Sspace$ we call  
\begin{equation}\label{e:leading c}
\lim_{x\to 0^+} \frac{\varphi'(x)}{\gamma x^{\gamma-1}} 
\end{equation}
the \emph{leading coefficient of $\varphi$ in $\Sspace$}. The meaning of this number is explained in Lemma~\ref{l:potential form} in \S \ref{ss:basic lemmas}.
For every $\varphi$ in $C_\dag^{1,\gamma}(\R)$ put
\begin{equation}
\label{d:norm}
|\varphi|_{1,\gamma} \= \sup_{x\in (0,1]} \frac{|\varphi'(x)|}{\gamma x^{\gamma-1}} \text{ and } \|\varphi\|_{1,\gamma} \= \|\varphi\| + |\varphi|_{1,\gamma}.
\end{equation}
The space $\Sspace$ endowed with the norm $\|\cdot\|_{1,\gamma}$ is a \textsc{Banach} space. 

We say that a potential $\varphi$ in $\Sspace$ undergoes a \emph{robust phase transition in temperature for~$f$ in~$\Sspace$}, if every  
potential sufficiently close to~$\varphi$ in~$\Sspace$ undergoes a phase transition in temperature. In this case,~$\gamma\le \alpha$ and the leading coefficient
of~$\varphi$ in~$\Sspace$ is nonpositive by Theorem~\ref{t:sarig}(3).
We say that~$\delta_0$ is, \emph{robustly in~$\Sspace$}, the unique maximizing measure for~$\varphi$, if every potential sufficiently close to~$\varphi$ in~$\Sspace$ has~$\delta_0$ as its unique maximizing measure.
The following result characterizes potentials in $\Sspace$ undergoing a robust phase transition in temperature. 
\begin{theo}
\label{t:robust sarig}
Let~$\alpha$  be in~$(0,+\infty)$ and let $f$ be the \textsc{Manneville–Pomeau} map of parameter~$\alpha$.
Let $\gamma$ be in $(0,\alpha]$ and let $\varphi$ be a potential in $\Sspace$.
The following are equivalent.
\begin{enumerate}
\item[1.] $\varphi$ undergoes a robust phase transition in temperature for $f$ in $\Sspace$;
\item[2.] $\delta_0$ is, robustly in $\Sspace$, the unique maximizing measure for $\varphi$; 
\item[3.] $\varphi$ undergoes a phase transition in temperature and its leading coefficient in~$\Sspace$ is strictly negative.
\end{enumerate}
\end{theo}
Observe that for all~$\gamma \in (0,\alpha]$,~$\gamma' \in (\gamma,\alpha]$, and~$\varphi \in C_\dag^{1,\gamma'}(\R)$, the leading coefficient of~$\varphi$ in~$\Sspace$ is zero. By Theorem~\ref{t:robust sarig}, no phase transition in temperature of a potential in~$C_\dag^{1,\gamma'}(\R)$ is robust in~$\Sspace$.
When~$\gamma < \alpha$, not all nonrobust phase transitions in temperature in~$\Sspace$ are of this form. For example, the potential~$\tvarphi$ defined by
\[
\tvarphi(x) \= \frac{1}{-x + \log x} \, x^\gamma \quad \text{for } x \in (0,1], \quad \text{and} \quad \tvarphi(0) \= 0,
\]
belongs to~$\Sspace \setminus \bigcup_{\gamma' \in (\gamma,\alpha]} C_\dag^{1,\gamma'}(\R)$, has zero leading coefficient in~$\Sspace$, and undergoes a phase transition in temperature (see Corollary~\ref{c:phase transitions}(1)).
The situation is entirely different in the case~$\gamma = \alpha$, as shown in the following result, which is an immediate consequence of Theorems~\ref{t:sarig alpha} and~\ref{t:robust sarig}.
\begin{coro}
\label{c:robust vs nonrobust}
Let~$\alpha$  be in~$(0,+\infty)$ and let~$f$ be the \textsc{Manneville–Pomeau} map of parameter~$\alpha$.
For every $\varphi$ in $C_\dag^{1,\alpha}(\R)$ the following are equivalent.
\begin{enumerate}
\item[1.] $\varphi$ undergoes a phase transition in temperature;
\item[2.] $\varphi$ undergoes a robust phase transition in temperature for $f$ in $C_\dag^{1,\alpha}(\R)$.
\end{enumerate}   
 When these equivalent conditions hold, the leading coefficient of~$\varphi$
in~$C_\dag^{1,\alpha}(\R)$ is strictly negative.
\end{coro}

We conclude this section by discussing the phase diagram of potentials in~$\Sspace$, and the role of Theorem~\ref{t:robust sarig} and Corollary~\ref{c:robust vs nonrobust} in its description.
Let~$\alpha$  be in~$(0,+\infty)$ and let~$f$ be the \textsc{Manneville–Pomeau} map of parameter~$\alpha$.
Given~$\gamma \in (0,+\infty)$, we define the \emph{phase transition in temperature locus}~$\cP\cT_{T}(\gamma)$ as the set of potentials in~$\Sspace$ that undergo a phase transition in temperature at~$\beta = 1$. By Theorem~\ref{t:sarig}, the set~$\cP\cT_{T}(\gamma)$ is empty for~$\gamma > \alpha$, and for~$\gamma \in (0,\alpha]$, by Proposition~\ref{c:tpt}, it has empty interior in~$\Sspace$.
We also define the \emph{robust phase transition in temperature locus}~$\cP\cT_{RT}(\gamma)$ as the subset of~$\cP\cT_{T}(\gamma)$ consisting of potentials whose phase transition in temperature is robust.
Motivated by the Gibbs phase rule from Statistical Mechanics (see, for instance,~\cite{Isr79}), one may ask about the regularity of the sets~$\cP\cT_{T}(\gamma)$ and~$\cP\cT_{RT}(\gamma)$ in~$\Sspace$. For example, by Theorem~\ref{t:robust sarig}, one may ask whether~$\cP\cT_{T}(\gamma)$ is a real-analytic manifold with boundary in~$\Sspace$, or whether~$\cP\cT_{RT}(\gamma)$ is a real-analytic manifold in~$\Sspace$.
By Corollary~\ref{c:robust vs nonrobust}, in the case~$\gamma = \alpha$, these questions reduce to: Is~$\cP\cT_{T}(\alpha)$ a real-analytic manifold in~$C^{1,\alpha}_\dag(\R)$?
In the companion paper~\cite{CorRiv25b}, we answer similar questions in a topological setting and for \textsc{Hölder} continuous potentials, but the questions above remain open.

\subsection{The Key Lemma and a dichotomy for \textsc{Hölder} continuous potentials}
\label{ss:about the proof}

The main new feature in Theorems~\ref{t:sarig} and~\ref{t:robust sarig} is the role
played by the invariant measure~$\delta_0$ as a maximizing measure for the
potential. To prove that, for a \textsc{Hölder} continuous potential with a phase
transition in temperature, the measure~$\delta_0$ is the unique maximizing
measure (Theorem~\ref{t:sarig}(3)), we use two ingredients.
The first ingredient is the following Key Lemma, which rules out this
possibility whenever there exists an invariant measure distinct from~$\delta_0$
with the same maximal integral.
\begin{generic}[Key Lemma] 
Let~$\alpha$  be in~$(0,+\infty)$ and let~$f$ be the \textsc{Manneville–Pomeau} map of parameter~$\alpha$.
For each \textsc{Hölder} continuous potential $\varphi$ and each  
$\mu$ in~$\sM$ distinct from~$\delta_0$, we have
\begin{equation}
\Pf(\varphi) > \int \varphi \dd \mu.
\end{equation}
\end{generic}
The Key Lemma was first proved for rational maps in \cite{InoRiv12} and then for multimodal maps in \cite{Li15}. Since none of these proofs applied directly to the setting of this article we have included a detailed and simplified proof of the Key Lemma for \textsc{Manneville–Pomeau} maps in \S\ref{s:key lemma}.

The second ingredient is a dichotomy for \textsc{Hölder} continuous potentials,
which we state in~\S\ref{ss:dichotomy} as Proposition~\ref{c:tpt} and derive from \cite[Theorem~A.2(2)]{InoRiv26}. It implies that if a \textsc{Hölder} continuous potential undergoes a phase transition in temperature at $\beta_*>0$,
then for all~$\beta\ge \beta_*$, the pressure satisfies
$P(\beta\phi)=\beta\phi(0)$. In \cite[Theorem~C]{BomFer23}, as in Proposition~\ref{c:tpt}, it was shown that for \textsc{Manneville–Pomeau}-like maps of the circle, \textsc{Hölder} continuous potentials exhibit at most one phase transition in temperature.

\subsection{Notes and references}
\label{ss:notes}

Although not very detailed, the first proof that the geometric potential for the~\textsc{Manneville–Pomeau} map undergoes a phase transition in temperature at~$1$ appeared in~\cite{PreSla92}. See also~\cite[Theorem~4.3]{CliTho13} for a proof when~$\alpha$ is in $(0,1)$. We take the opportunity to give a simple proof of this fact in~\S\ref{ss:Applications}. 

Let~$\alpha$  be in~$(0,+\infty)$ and let~$f$ be the \textsc{Manneville–Pomeau} map of parameter~$\alpha$.
Corollary~\ref{c:phase transitions}(1) in \S\ref{ss:Applications} asserts that for a \textsc{Hölder} continuous potential~$\varphi$ satisfying $\varphi(0) = 0$, if there exists~$C > 0$ such that~$\varphi \leq C\omega_\alpha$ on~$[0,1]$, then~$\varphi$ undergoes a phase transition in temperature. This naturally raises the question of whether every \textsc{Hölder} continuous potential~$\varphi$ with $\varphi(0) = 0$ that undergoes a phase transition in temperature necessarily satisfies this condition. As the following example shows, the answer is negative.
Let~$x_1$ denote the discontinuity point of the map~$f$. Consider the potential~$\psi$ defined by~$\psi(x) \= -x^\alpha (x - x_1)^2$. By Theorem~\ref{t:sarig}(2), this potential undergoes a phase transition in temperature, yet there does not exist~$C > 0$ such that~$\psi \leq C\omega_\alpha$ on~$[0,1]$.
However, one may still ask whether there are $C > 0$ and a bounded measurable function $h$ such that
$$\varphi \le C\omega_\alpha + h\circ f - h \text{ on } [0,1].$$
In this case, the occurrence of a phase transition in temperature would be equivalent to the latter property.

A naturally related question, for every $\gamma \in (0,\alpha]$, is the following: Let~$\varphi \in \Sspace$ be a potential satisfying $\varphi(0) = 0$, with a negative leading coefficient in~$\Sspace$, and for which~$\delta_0$ is the unique maximizing measure. Is~$\varphi$ cohomologous, via a bounded measurable function, to a potential~$\phi \in \Sspace$ that is negative on~$(0,1]$, satisfies~$\phi(0) = 0$, and has a negative leading coefficient in~$\Sspace$?
A positive answer to this question would provide an alternative proof of Theorem~\ref{t:sarig}(2); see Corollary~\ref{c:phase transitions}(2).

In~\cite[Proposition~1]{Sar01a}, \textsc{Sarig} used the following variant of pressure,
\begin{equation}
  \label{eq:Sarig-pressure}
  \tP(\varphi)
  \=
  \sup\left\{ h_{\mu} + \int \varphi \dd \mu \colon \mu \in \sM,\ \mu(\{0\}) = 0 \right\}.
\end{equation}
For every continuous potential ${\varphi \colon [0,1] \to \R}$, this number coincides with~$P(\varphi)$.
Clearly ${\tP(\varphi) \le P(\varphi)}$, so it is enough to prove the reverse inequality.
We first recall why the supremum in the definition of~$P(\varphi)$ is attained. By Lemma~\ref{l:extension}(5), the entropy map of~$f$ is upper-semicontinuous on~$\sM$; since~$\sM$ is compact and the map~$\nu \mapsto \int \varphi \dd\nu$ is continuous, the map
$$
\nu \mapsto h_\nu+\int \varphi \dd\nu
$$
is upper-semicontinuous on~$\sM$ and therefore attains its maximum.
Let~$\mu$ in~$\sM$ be such that
\begin{equation}
  \label{eq:Sarig-pressure-equilibrium}
  P(\varphi)
  =
  h_{\mu} + \int \varphi \dd \mu.
\end{equation}
By the ergodic decomposition theorem, since the entropy and the integral of~$\varphi$ are affine with respect to the measure, we may replace~$\mu$ by one of its ergodic components that still realizes the maximum. Thus, suppose~$\mu$ is ergodic.
Denote, for each~$n$ in~$\N$, by~$p_n$ the periodic point of~$f$ of period~$n$ in~$(x_n,x_{n-1}]$ and put
\begin{equation}
  \label{eq:Sarig-periodic-measure}
  \mu_n
  \=
  \frac{1}{n}(\delta_{p_n}+\cdots+\delta_{f^{n-1}(p_n)}).
\end{equation}
Then~$\mu_n$ belongs to~$\sM$, satisfies~$\mu_n(\{0\})=0$ and~$h_{\mu_n}=0$, and converges weakly to~$\delta_0$ as~$n\to +\infty$.
If~$\mu(\{0\})=0$, then
\begin{equation}
  \label{eq:Sarig-pressure-nonatomic}
  \tP(\varphi)
  \ge
  h_{\mu}+\int \varphi \dd \mu
  =
  P(\varphi).
\end{equation}
On the other hand, if~$\mu(\{0\})>0$, then, since~$\mu$ is ergodic, one has~$\mu=\delta_0$ and~$P(\varphi)=\varphi(0)$.
Thus
\begin{equation}
  \label{eq:Sarig-pressure-atomic}
  \tP(\varphi)
  \ge
  \limsup_{n\to +\infty}\left(h_{\mu_n}+\int \varphi \dd \mu_n\right)
  =
  \int \varphi \dd \delta_0
  =
  P(\varphi).
\end{equation}
This completes the proof of $\widetilde P(\varphi)=P(\varphi)$.

\subsection{About the proof of the theorems and the organization}
\label{ss:organization}

We close the introduction with a brief guide to the structure of the paper and to
the proof of the main results.

The central technical result of the article is the Main Theorem. It is stated in
\S\ref{s:main}, proved in \S\ref{s:proofs}, and used in
\S\ref{ss:proofs} to derive Theorems~\ref{t:sarig}, \ref{t:sarig alpha},
and~\ref{t:robust sarig}. Roughly speaking, it shows that if
$\varphi \in \Sspace$ has negative leading coefficient, then $\varphi$
undergoes a phase transition in temperature if and only if $\delta_0$ is the
unique maximizing measure for~$\varphi$. Moreover, when these equivalent
conditions hold, they define an open subset of~$\Sspace$.

The proof of the Main Theorem relies on several auxiliary ingredients developed
earlier in the paper. In \S\ref{ss:indifferent-branch}, we prove geometric
estimates near the indifferent fixed point, namely
Lemmas~\ref{l:indifferent-branch'} and~\ref{l:geometric-distortion}. In
\S\ref{ss:basic lemmas}, we prove Lemmas~\ref{l:sarig class is loc holder}
and~\ref{l:potential form}; the latter describes the form of potentials
in~$\Sspace$ near~$0$ and clarifies the meaning of the leading coefficient
introduced in \S\ref{ss:Robust phase transitions in temperature}. In
\S\ref{ss:extension}, we give a detailed proof of the doubling construction for
\textsc{Manneville--Pomeau} maps, and in \S\ref{ss:pressure formulas} we use
this construction to derive the pressure formulas in
Proposition~\ref{l:top pressure}. We finish this section with the dichotomy for \textsc{Hölder} continuous potentials stated in \S\ref{ss:dichotomy} as
Proposition~\ref{c:tpt}. In \S\ref{s:key lemma}, we prove the Key Lemma, and finally, in \S\ref{s:inducing}, we introduce the
inducing scheme and prove the Bowen-Type Formula. Together with
Proposition~\ref{l:operator}, this formula is used to prove
Proposition~\ref{p:omega transitions}; combined with
Proposition~\ref{c:tpt}, it yields Corollaries~\ref{c:omega-phase-transition}
and~\ref{c:no phase transition}.

The most delicate parts of the paper are Theorem~\ref{t:sarig}(2) and the
implication \(3 \Rightarrow 1\) in Theorem~\ref{t:robust sarig}. Both follow
from the Main Theorem together with Remark~\ref{r:n0}. On the other hand,
Theorem~\ref{t:sarig}(3), namely the uniqueness of $\delta_0$ as a maximizing
measure under the assumption of a phase transition in temperature, is obtained
by combining Proposition~\ref{c:tpt} with the Key Lemma: If $\delta_0$ is not
the unique maximizing measure for~$\varphi$, then the same is true for
$\beta\varphi$ for every $\beta>0$; the Key Lemma gives
$P(\beta\varphi)>\beta\varphi(0)$ for every $\beta>0$, and
Proposition~\ref{c:tpt} then rules out the existence of a phase transition in
temperature.

\medskip
\noindent\textbf{Logical dependence of the results.}
For the convenience of the reader, we summarize the main logical dependencies
in three layers.

\smallskip
\noindent\emph{First layer: Preliminary estimates and induced pressure.}
The doubling
construction in Lemma~\ref{l:extension}, Lemmas~\ref{l:inverse branches}
and~\ref{l:generating partition} lead to the pressure formulas in
Proposition~\ref{l:top pressure}. Together with Proposition~\ref{l:operator}, these formulas
are used in the proof of the Bowen-Type Formula. Proposition~\ref{l:top pressure} is also used to prove Proposition \ref{c:tpt} and Corollary \ref{c:phase transitions}.  Proposition~\ref{p:omega transitions},
whose proof relies on the Bowen-Type Formula and
Proposition~\ref{l:operator}, is then combined with Proposition~\ref{c:tpt}
to yield Corollaries~\ref{c:omega-phase-transition} and~\ref{c:no phase transition}.

\[
\begin{tikzcd}[column sep=3.6em,row sep=2.8em]
\text{Lemmas~\ref{l:extension}, \ref{l:inverse branches}, and~\ref{l:generating partition}}
\arrow[r] &
\text{Proposition~\ref{l:top pressure}}
\arrow[dr] & \\
&
\text{Proposition~\ref{l:operator}}
\arrow[r] &
\text{Bowen-Type Formula}
\end{tikzcd}
\]

\smallskip
\noindent\emph{Second layer: The Main Theorem.}
The implication \(1 \Rightarrow 2\) in the Main Theorem uses
Proposition~\ref{c:tpt} and the Key Lemma. The implication \(2 \Rightarrow 3\)
is a compactness argument. The implication \(3 \Rightarrow 1\) uses
Lemma~\ref{l:sarig class is loc holder}, Proposition~\ref{l:operator}, and the
Bowen-Type Formula. The proof of the Key Lemma, given in \S\ref{s:key lemma},
relies on Proposition~\ref{p:good IFS}.

\[
\begin{tikzcd}[column sep=3.6em,row sep=2.8em]
\text{Proposition~\ref{p:good IFS}}
\arrow[r] &
\text{Key Lemma}
\arrow[r] &
(1)\Rightarrow(2)\text{ in the Main Theorem}
\end{tikzcd}
\]

\[
\begin{tikzcd}[column sep=3.6em,row sep=2.8em]
\text{Lemma~\ref{l:sarig class is loc holder}}
\arrow[dr] &
\text{Proposition~\ref{l:operator}}
\arrow[d] &
\text{Bowen-Type Formula}
\arrow[dl] \\
&
(3)\Rightarrow(1)\text{ in the Main Theorem} &
\end{tikzcd}
\]

\smallskip
\noindent\emph{Third layer: Applications to the main theorems.}
Theorem~\ref{t:sarig}(1) follows from~\eqref{eq:10} and
Lemma~\ref{l:potential form}. Theorem~\ref{t:sarig}(2), as well as the
implication \(3 \Rightarrow 1\) in Theorem~\ref{t:robust sarig}, follow from
the Main Theorem and Remark~\ref{r:n0}. Theorem~\ref{t:sarig}(3) is obtained
from Proposition~\ref{c:tpt}, the Key Lemma,
Proposition~\ref{p:omega transitions}, and Lemma~\ref{l:potential form}.
Theorem~\ref{t:sarig alpha} uses Theorem~\ref{t:sarig}(3),
\eqref{eq:10}, Lemma~\ref{l:potential form}, Proposition~\ref{l:operator}, the
Bowen-Type Formula, and Proposition~\ref{c:tpt}. For
Theorem~\ref{t:robust sarig}, the implication \(1 \Rightarrow 2\) follows from
Theorem~\ref{t:sarig}(3), whereas \(2 \Rightarrow 3\) follows from
Theorem~\ref{t:sarig}(1) and~(2). Corollary~\ref{c:phase transitions} follows
from Proposition~\ref{p:omega transitions} and Lemma~\ref{l:potential form}.
Finally, Corollary~\ref{c:sarig} follows from Theorems~\ref{t:sarig}
and~\ref{t:sarig alpha}, and Corollary~\ref{c:robust vs nonrobust} follows
from Theorems~\ref{t:sarig alpha} and~\ref{t:robust sarig}.

\section{Preliminaries}
\label{s:preliminaries}

We use~$\N$ to denote the set of integers greater than or equal to~$1$ and~$\N_0 \= \N \cup \{ 0 \}$.
Fix~$\alpha$ in~$(0, +\infty)$ throughout the rest of this paper.

\subsection{Indifferent branch and bounded distortion}
\label{ss:indifferent-branch}
Note that the function~$\log Df$ is strictly increasing and \textsc{H{\"o}lder} continuous of exponent~$\min \{1, \alpha\}$.
 Denote by~$x_1$ the unique discontinuity of~$f$ .
Then ${f(x_1) = 1}$ and the map ${f \colon [0, x_1] \to [0, 1]}$ is a diffeomorphism.
Put ${x_0 \= 1}$ and for each integer~$j$ satisfying ${j \ge 2}$, put~$x_j \= f|_{[0, x_1]}^{-(j - 1)}(x_1)$.
Also put, for every~$n$ in $\N_0$, 
$$J_n\=(x_{n+1},x_n].$$  
For every~$j$ in~$\N_0$ note that~$f(x_{j + 1}) = x_j$.

The following limit is known to the specialists. However, we prefer to state and prove it as part of 
Lemma~\ref{l:indifferent-branch'} for the reader's convenience.

\begin{lemm}
  \label{l:indifferent-branch'}
  We have
  \begin{equation}
    \label{eq:10}
    \lim_{n \to +\infty} n \cdot x_n^{\alpha}
    =
    \frac{1}{\alpha}.
  \end{equation}
  Moreover, there is~$\varepsilon_0$ in~$(0, 1)$ such that for all~$n$ in~$\N$ and~$x$ in~$J_n$ we have
  \begin{equation}
    \label{eq:11}
    (1 + \varepsilon_0 n)^{\frac{1}{\alpha} + 1}
    \le
    Df^n(x)
    \le
    (1 + \varepsilon_0^{-1} n)^{\frac{1}{\alpha} + 1}.
  \end{equation}
\end{lemm}

\begin{proof}
  Following~\cite[\S11]{Fat19}, we consider the maps~$\whf$, $h$, and~$\sF$ from~$(0, +\infty)$ to~$(0, +\infty)$ defined by
  \begin{equation}
    \label{eq:12}
    \whf(x)
    \=
    x(1 + x^{\alpha}),
    h(x)
    \=
    x^{-\alpha},
    \text{ and }
    \sF
    \=
    h \circ \whf \circ h^{-1}.
  \end{equation}
  For every~$X$ in~$(0, +\infty)$, we have
  \begin{equation}
    \label{eq:13}
    \sF(X)
    =
    X (1 + X^{-1})^{-\alpha}.
  \end{equation}
  It follows that there is~$C$ in~$(0, +\infty)$, such that for every sufficiently large~$X$ we have
  \begin{equation}
    \label{eq:14}
    |\sF(X) - (X - \alpha)|
    \le
    C \frac{1}{X}.
  \end{equation}
  
  For every~$n$ in~$\N_0$ put ${X_n \= h(x_n)}$.
  Since~$\whf$ coincides with~$f$ on~$[0, x_1]$, for every~$n$ in~$\N_0$ we have
  \begin{equation}
    \label{eq:15}
    \sF(X_{n + 1})
    =
    X_n.
  \end{equation}
  Since the sequence~$(x_n)_{n \in \N_0}$ is strictly decreasing and the only fixed point of~$f$ on~$[0, x_1]$ is~$0$, we have
  \begin{equation}
    \label{eq:16}
    \lim_{n \to +\infty} x_n = 0
    \text{ and }
    \lim_{n \to +\infty} X_n = +\infty.
  \end{equation}
  Together with~\eqref{eq:14} and~\eqref{eq:15}, this implies that for every sufficiently large~$n$ we have~${X_{n + 1} \ge X_n + 2\alpha/3}$.
  It follows that for every sufficiently large~$n$, we have
  \begin{equation}
    \label{eq:17}
    X_n
    \ge
    \alpha n / 2.
  \end{equation}
  Next, define for each~$n$ in~$\N$ the number ${\delta_n \= X_n - \alpha n}$.
  By~\eqref{eq:14} and~\eqref{eq:17}, for every sufficiently large~$n$ we have
  \begin{equation}
    \label{eq:18}
    |\delta_{n + 1} - \delta_n|
    \le
    (2C\alpha^{-1}) \frac{1}{n}.
  \end{equation}
  It follows that there is a constant~$C'$ in~$(0, +\infty)$ such that for every sufficiently large~$n$ we have
  \begin{equation}
    \label{eq:19}
    |X_n - \alpha n|
    =
    |\delta_n|
    \le
    C' \log n.
  \end{equation}
  This implies~\eqref{eq:10}.
  To prove~\eqref{eq:11}, note that~$Df$ is increasing, so for all~$n$ in~$\N$ and~$x$ in~$J_n$ we have
  \begin{multline}
    \label{eq:20}
    \prod_{j = 1}^{n} \left( 1 + (1 + \alpha) \frac{1}{X_{j + 1}} \right)
    =
    Df^n(x_{n + 1})
    \le
    Df^n(x) \\
    \le
    Df^n(x_n)
    =
    \prod_{j = 1}^n \left( 1 + (1 + \alpha) \frac{1}{X_j} \right).
  \end{multline}
  By~\eqref{eq:18}, there is~$C''$ in~$(0, +\infty)$ such that for every sufficiently large~$j$ we have
  \begin{equation}
    \label{eq:21}
    \left| \log \left( 1 + (1 + \alpha) \frac{1}{X_j} \right) - \left( \frac{1}{\alpha} + 1 \right) \frac{1}{j} \right|
    \le
    C'' \frac{\log j}{j^2}.
  \end{equation}
  Combined with~\eqref{eq:20}, this implies that there is~$C'''$ in~$(1, +\infty)$ such that for every sufficiently large~$n$ we have
  \begin{equation}
    \label{eq:22}
    \frac{1}{C'''} n^{\frac{1}{\alpha} + 1}
    \le
    Df^n(x)
    \le
    C''' n^{\frac{1}{\alpha} + 1}.    
  \end{equation}
  Together with the fact that for every~$n$ in~$\N$ we have ${Df^n(x_n) > 1}$, this implies~\eqref{eq:11}.
\end{proof}

\begin{lemm}[Bounded distortion]
  \label{l:geometric-distortion}
  There are~$\varepsilon_1$ in $(0,1)$ and~$C_1$ in~$(1, +\infty)$, such that the following properties hold.
  For every~$n$ in~$\N$ and every connected component~$J$ of~$f^{-n}(J_0)$, we have
  \begin{equation}
    \label{eq:23}
    |J|
    \le
    \frac{|J_0|}{(1 + \varepsilon_1 n)^{\frac{1}{\alpha} + 1}}
  \end{equation}
  and for all~$x$ and~$y$ in~$J$ we have
  \begin{equation}
    \label{eq:24}
    Df^n(x)
    \ge
    (1 + \varepsilon_1 n)^{\frac{1}{\alpha} + 1}
    \text{ and }
    C_1^{-1}
    \le
    \frac{Df^n(x)}{Df^n(y)}
    \le
    C_1.
  \end{equation}
\end{lemm}
\proof
 Let~$\varepsilon_0$ in~$(0, 1)$ be from Lemma~\ref{l:indifferent-branch'} and let~$\varepsilon_1$ in~$(0, \varepsilon_0]$ be such that for every~$x$ in~$J_0$ we have
  \begin{equation}
    \label{eq:25}
    Df(x)
    \ge
    (1 + \varepsilon_1)^{\frac{1}{\alpha} + 1}.
  \end{equation}
  Put
  \begin{equation}
    \label{eq:26}
    C_1'
    \=
    |J_0| \cdot |Df|_{\min \{1, \alpha\}} \sum_{m = 0}^{+\infty} \frac{1}{(1 + \varepsilon_1 n)^{\min \{1, \alpha\}\left( \frac{1}{\alpha} + 1 \right)}}
  \end{equation}
  and note that~$C_1'$ is finite because ${\min \{1, \alpha\} \left( \frac{1}{\alpha} + 1 \right) > 1}$.

  Let~$n$ in~$\N$ be given and let~$J$ be a connected component of~$f^{-n}(J_0)$.
  For each~$m$ in ${\{0, \ldots, n - 1\}}$, let~$a_m$ be equal to~$1$ if~${f^m(J) \subseteq J_0}$ and to~$0$ otherwise.
  Denote by~$r$ the number of~$1$'s in the sequence $a_0 \cdots a_m$, by~$s$ the number of blocks of~$0$'s, and by~$\ell_0$, \ldots, $\ell_s$ the lengths of the blocks of~$0$'s.
  We thus have ${r + \ell_1 + \cdots + \ell_s = n}$.
  Combining~\eqref{eq:11} and~\eqref{eq:25}, we obtain for every~$x$ in~$J$
  \begin{equation}
    \label{eq:27}
    Df^n(x)
    \ge
    (1 + \varepsilon_1)^{r \left( \frac{1}{\alpha} + 1 \right)} \prod_{k = 1}^s (1 + \varepsilon_1 \ell_k)^{\frac{1}{\alpha} + 1}
    \ge
    (1 + \varepsilon_1 n)^{\frac{1}{\alpha} + 1}.
  \end{equation}
  This proves the first inequality in~\eqref{eq:24} and implies~\eqref{eq:23}.
  Using the chain rule, \eqref{eq:23}, and the definition~\eqref{eq:26} of~$C_1'$, we obtain for all~$x$ and~$y$ in~$J$
  \begin{multline}
    \label{eq:28}
    \left| \log \frac{Df^n(x)}{Df^n(y)} \right|
    \le
    \sum_{m = 0}^{n - 1} |\log Df(f^m(x)) - \log Df(f^m(y))| 
    \\
    \le
    |Df|_{\min \{1, \alpha\}} \sum_{m = 0}^{n - 1} |f^m(J)|^{\min \{1, \alpha\}}
     \le
    C_1'.
  \end{multline}
  This proves~\eqref{eq:24} with ${C_1 = \exp(C_1')}$.
  \endproof

\subsection{Bounded variation ergodic sums and expansion at zero for potential in $\Sspace$}
\label{ss:basic lemmas}

The next lemma is about the control up to an additive constant of the ergodic sums. This is a classical technical requirement in the theory of Thermodynamic Formalism that is satisfied, for example, by the geometric potential.
For a potential~$\varphi$ and an integer~$n\ge 1$, write
$S_n\varphi\=\sum_{j=0}^{n-1}\varphi\circ f^j$.
\begin{defi}
\label{d:bd potentials}
We say that a continuous potential $\varphi:[0,1]\to \R$ has \emph{bounded variation ergodic sums on $J_0$ for $f$}, if there is a constant $C>0$ such that for every $n$ in $\N$, every connected component $J$ of $f^{-n}(J_0)$ and all $x$ and $y$ in $J$ the following inequality holds:
\begin{equation}
|S_n\varphi(x)-S_n\varphi(y)|\le C.
\end{equation}
\end{defi}

\begin{lemm}\label{l:sarig class is loc holder}
Let~$\gamma$ be a positive real number. There is a positive constant~$D$ such that every potential~$\varphi$ in~$\Sspace$ has bounded variation ergodic sums on~$J_0$ for~$f$ with constant~$D|\varphi|_{1,\gamma}$.
\end{lemm}

\proof
Let $\gamma$ be in $(0,+\infty)$, and let $\varepsilon_1 \in (0,1)$ be given by Lemma  \ref{l:geometric-distortion}. 
Put $\gamma_0\=\min\{1,\gamma\}$.
By \eqref{eq:10}, there is  a constant $K \ge 1$ such that for every $n$ in $\N$ we have
\begin{equation}
\label{e:100}
x_{n}^{\gamma_0-1} \le K \frac{1}{n^{\frac{\gamma_0-1}{\alpha}}}.
\end{equation}
Put 
\begin{equation}
D \=  K \gamma \varepsilon_1^{-(\alpha^{-1}+1)}\sum_{j=1}^{+\infty}  \frac{1}{j^{1+ \frac{\gamma_0}{\alpha}}}.
\end{equation} 
Fix $\varphi$ in $\Sspace$, an integer $n\ge 1$, a connected component $J$ of $f^{-n}(J_0)$, and $x$ and $y$ in $J$. 
Observe that for every $j$ in $\{0,\ldots, n-1\}$ we have $f^j(J)\subseteq (x_{n-j+1}, 1]$. 
Together with \eqref{e:100}, this implies that for $j$ in $\{0,\ldots, n-1\}$ we have
\begin{equation}
\sup_{u\in f^j(J)} |\varphi'(u)| \le  |\varphi|_{1,\gamma} \sup_{u\in f^j(J)} \gamma u^{\gamma-1}
 \le |\varphi|_{1,\gamma} \frac{   K \gamma }{(1+n-j)^{\frac{\gamma_0-1}{\alpha}}}.
 \end{equation}
Using 
Lemma \ref{l:geometric-distortion}\eqref{eq:23} we get
\begin{equation}
\left(\sup_{u\in f^j(J)} |\varphi'(u)|\right) |f^j(x) - f^j(y)| \le|\varphi|_{1,\gamma}  K\gamma\varepsilon_1^{-(\alpha^{-1}+1)} \frac{1}{(1+n-j)^{1+\frac{\gamma_0}{\alpha}}}.
\end{equation}
Therefore,
\begin{equation}
\begin{split}
 |S_n\varphi(x)-S_n\varphi(y)| & \le \sum_{j=0}^{n-1} |\varphi(f^j(x))-\varphi(f^j(y))| 
 \\
 & \le   \sum_{j=0}^{n-1} \left(\sup_{u\in f^j(J)} |\varphi'(u)|\right) |f^j(x) - f^j(y)| 
\\
& \le  D|\varphi|_{1,\gamma}.
\end{split}
\end{equation} 
This concludes the proof of the lemma.
\endproof

Now, we state a lemma concerning the expansion at zero for potentials in $\Sspace$. In particular, this lemma clarifies the meaning of the leading coefficient introduced in \S\ref{ss:Robust phase transitions in temperature}. The  proof is quite simple, and we included it for completeness.

\begin{lemm}
\label{l:potential form}
Let $\gamma$ be in $(0,+\infty)$ and let $\varphi \: [0,1]\to \R$ be a function. We have that $\varphi$ belongs to $\Sspace$ if and only if there are $c$ in $\R$ and a function 
$h$ in $C_\dag^1(\R)$ such that  for every $x$ in $[0,1]$ 
\begin{equation} 
\varphi(x)=\varphi(0)+cx^\gamma+h(x)x^\gamma, 
h(0)=0, \text{ and } \lim_{x\to 0^+}h'(x)x =0.
\end{equation}
In particular, $c$ and $h$ are unique with $c=\lim_{x\to 0^+}\frac{\varphi'(x)}{\gamma x^{\gamma-1}}$. 
\end{lemm}

\proof
If $\varphi$ belongs to $\Sspace$, put $c\=\lim_{x\to 0^+}\frac{\varphi'(x)}{\gamma x^{\gamma-1}}$, and define the function $h\:(0,1]\to \R$ by
$$
h(x)\= \frac{\varphi(x)-\varphi(0)}{x^\gamma} -c. 
$$
By the definition of $c$, for every $\varepsilon>0$, there is $\delta>0$ such that for every $x$ in $(0,\delta)$ we have
\begin{equation}
(c-\varepsilon)\gamma x^{\gamma-1} \le \varphi'(x) \le (c+\varepsilon) \gamma x^{\gamma-1}.
\end{equation}
Then,
\begin{equation}
(c-\varepsilon) x^{\gamma} \le \varphi(x)-\varphi(0) \le (c+\varepsilon) x^{\gamma}.
\end{equation}
It follows that 
\begin{equation}
\label{e: h at 0}
\lim_{x\to 0^+} h(x) = 0,
\end{equation}
so the function $h$ extends continuously to $[0,1]$. Denote the extension by $h$ and note that $h(0)=0$. Together with the definition of $h$, for every $x$ in $[0,1]$, we get that
\begin{equation}
\label{e: potential expansion}
\varphi(x)=\varphi(0)+cx^\gamma+h(x)x^\gamma.
\end{equation}
Since $h$ is continuously differentiable on $(0,1]$, we have that $h$  is in $C_\dag^1(\R)$. Finally, from \eqref{e: h at 0} and \eqref{e: potential expansion}, we get that
\begin{equation}
\lim_{x\to 0^+}h'(x)x =0.
\end{equation}
The reverse implication follows by direct computation. So, the lemma is proved.
\endproof

\subsection{Continuous extension and topological exactness}
\label{ss:extension}
In this subsection, we construct a continuous extension of the dynamical system $([0,1],f)$, which allows the application of the theory of thermodynamic formalism on compact metric spaces. We introduce the following notation for continuous dynamical systems on compact metric spaces. Given a compact metric space $Z$ and a continuous transformation $T\: Z\to Z$, we denote by $\sM_T$ the space of \textsc{Borel} probability measures invariant by $T$. For  every $\mu$ in $\sM_T$ we denote by $h_\mu(T)$ the entropy of $\mu$, and given a continuous potential $\psi\:Z\to \R$
we denote by $P_T(\psi)$ the pressure for the potential $\psi$ of $T$. 

\begin{lemm}
\label{l:extension}
Put~$Y\=\bigcup_{n\in \N_0} f^{-n}(x_1)$. There is a totally ordered set~$X$, endowed with the order topology, and a continuous surjective map~$\pi\: X\to [0,1]$ such that the following hold.
\begin{itemize}
\item[1.] The sets $X\setminus \pi^{-1}(Y)$  and $ [0,1]\setminus Y$ are equal as totally ordered sets, the map $\pi$ is the identity on $X\setminus \pi^{-1}(Y)$,   and $\pi^{-1}(Y)$ consists of two disjoint copies $Y^-$ and $Y^+$ of $Y$ such that 
for every  $y$  in  $Y$ one has $\pi^{-1}(\{y\})=\{y^-,y^+\}$ and   $y^- < y^+$.
\item[2.] The order topology on $X$ is compact and metrizable.
\item[3.] There is continuous map $\tmpf\:X\to X$ such that 
 $\pi \circ \tmpf = \mpf \circ \pi$ \textcolor{black}{on ${X \setminus \{ x_1^+ \}}$,}
 \[
Y^-=\bigcup_{n\in \N} \tmpf^{-n}(1) \setminus \{1\} = \bigcup_{n\in \N} \tmpf^{-n}(x_1^-), \text{ and }Y^+=\bigcup_{n\in \N} \tmpf^{-n}(0)\setminus \{0\} = \bigcup_{n\in \N} \tmpf^{-n}(x_1^+).
\]
\item[4.] Put~$\tP_0\= [0,x_1^-]$ and~$\tP_1\=[x_1^+,1]$. The map~$\tpi\:X\to \{0,1\}^{\N_0}$ given  by $(\tpi (x))_k=i$ 
if $\tmpf^k(x)\in \tP_i$ is a topological conjugacy from
 $(X,\tmpf)$  to the full shift $(\{0,1\}^{\N_0},\sigma)$.
\item[5.] For every invariant measure $\nu$ for $\mpf$ there is unique invariant measure $\mu$ for $\tmpf$ such that $\pi_* \mu =\nu$ and the systems $(X,\tmpf,\mu)$ and  $([0,1],f,\nu)$ are isomorphic in measure. In particular, the map $\pi_*\:\sM_{\tmpf}\to \sM$ is a homeomorphism and the measure-theoretic entropy map of~$f$ is upper-semicontinuous.
\end{itemize}
\end{lemm}
The proof of Lemma \ref{l:extension} is given after the following couple of lemmas. 
Put~$\cP\=\{[0,x_1],(x_1,1]\}$ and denote by $g_0$ and $g_1$  the inverse branches of $f$ on $[0,x_1]$ and $(x_1,1]$, respectively. 
In general, given two partitions $\sP$ and $\sQ$ of some set $Z$ we denote by $\sP\vee \sQ$ the refinement of the both partitions defined by 
\begin{equation}
\sP\vee \sQ \= \{P\cap Q \colon P\in \sP,Q\in \sQ\}.
\end{equation}
\begin{lemm} 
\label{l:inverse branches} 
For all $x$  in $(0,1]$, $n$ in $\N$ and $y\in \mpf^{-n}(x)$ there is a unique inverse branch $\phi$ of $\mpf^n|_{(0,1]}$ defined on $(0,1]$  such that $y\in \phi((0,1])$.  For every $n$ in $\N$ there is a unique inverse branch $\phi$ of $\mpf^n$ defined on $[0,1]$  such that $\phi([0,1])=[0,x_n]$.
In particular, for every $n$ in~$\N$ and every $Q$ in $\bigvee_{k=0}^{n-1} \mpf^{-k}\cP$ we have that $f^n$ sends $Q$ diffeomorphically onto $(0,1]$ except when $0$ belongs to $Q$ in which case $f^n$ sends $Q$ diffeomorphically onto $[0,1]$. 
\end{lemm}
\proof
For every $n$ in $\N$, every inverse branch of $\mpf^n$ has the form $g_{i_1}\circ \cdots \circ g_{i_n}$. When at least one of the $g_{i_j}$ is equal to $g_1$, the domain of the inverse branch is $(0,1]$ and the image is an interval of the form $(a,b]$ with $a$ and $b$ in  
$\bigcup_{k=0}^{n} f^{-k}(1)$. 
When all the $g_{i_j}$ are equal to $g_0$, the domain of the inverse branch is $[0,1]$ and the image  is $[0,x_n].$
In particular, the images of two different inverse branches of $\mpf^n$ are disjoint, and the union of the images of all inverse branches of 
$\mpf^n$ covers $[0,1]$. On the other hand, since $\{0\}=\mpf^{-1}(0)$, for all 
$x$ in $(0,1]$,  $n$ in $\N$ and  $y$ in $\mpf^{-n}(x)$ we have $y\neq 0$. Therefore, there is a unique inverse branch $\phi$ of $\mpf^n|_{(0,1]}$ defined on $(0,1]$  such that $y\in \phi((0,1])$, and there is a unique inverse branch of $\mpf^n$ defined on $[0,1]$ whose image is 
$[0,x_n]$. 
\endproof

Let $Z$ be a topological space and $T\: Z \to Z$ be a map. One says that $T$ is \emph{topologically exact} if for every $z$ in $Z$ and every neighborhood $V$ 
of $z$ there is $k$ in $\N$ such that $T^k(V)=Z$.

\begin{lemm}
\label{l:generating partition} We have that 
\begin{equation}
\label{e:generating}
 \max_{Q \in \bigvee_{k=0}^{n-1} \mpf^{-k}\cP} \diam{Q} \to 0
\end{equation}
as $n$ goes to $+\infty$. In particular, the map $f$ is topologically exact on $(0,1]$.
\end{lemm}
\proof 
Notice that by \eqref{eq:10}  there is $C>0$ such  that for every $n$ in $\N$ one has
\begin{equation}
\label{e:x_n}
x_n \le C n^{-\frac{1}{\alpha}}.
\end{equation}
Let $n$ be in $\N$ and let  $Q$  be in $\bigvee_{k=0}^{n-1} \mpf^{-k}\cP$. By Lemma \ref{l:inverse branches},  
there is $i_1\cdots i_n$ in $\{0,1\}^n$ such that  $Q=g_{i_1}\circ \cdots \circ g_{i_n}((0,1])$ or $Q=g_0^n([0,1])=[0,x_n]$.
If there is an integer~$j\in  [\frac{n}{2},n]$ such that~$i_j=1$ then by Lemma \ref{l:geometric-distortion} we have 
\begin{equation}
\label{e:diam Q}
\diam{Q} \le |J_0|(1+\varepsilon_1 j )^{-\left(\frac{1}{\alpha}+1\right)}
 \le |J_0|\left(1+\varepsilon_1 \left\lceil \frac{n}{2} \right\rceil \right)^{-\left(\frac{1}{\alpha}+1\right)} .
 \end{equation}
 If for every integer~$j\in  [\frac{n}{2},n]$ we have that~$i_j=0$ then
 \begin{equation}
 f^{\lfloor \frac{n}{2} \rfloor}(Q)\subseteq [0,x_{n-\lceil \frac{n}{2} \rceil}].
 \end{equation}
By \eqref{e:x_n} and the fact  for every $x$ in $[0,1]$ one has $Df(x)\ge 1$ we get that 
 \begin{equation}
 \diam{Q} \le \diam{f^{\lfloor \frac{n}{2} \rfloor}(Q)} \le  x_{n-\lceil \frac{n}{2}\rceil} \le C \left \lfloor \frac{n}{2} \right\rfloor^{-\frac{1}{\alpha}}.
 \end{equation}
 Together with \eqref{e:diam Q}, this implies the first part of the lemma. 
 
 Now, we prove that $f$ is topologically exact on $(0,1]$. Observe that, by \eqref{e:generating}, for every open set $V$ contained in $(0,1]$ there are $n$ in $\N$ and $Q$ in $\bigvee_{k=0}^{n-1} \mpf^{-k}\cP$ such that $Q\setminus\{0\}\subseteq V$.
 By Lemma \ref{l:inverse branches},  we have $(0,1] = f^n(Q\setminus\{0\}) \subseteq f^n(V)$, which implies that $f$ is topologically exact on $(0,1]$ and concludes the proof of the lemma.
\endproof

\proof[Proof of Lemma \ref{l:extension}]
 The idea of the construction of the extension is well known. We follow  \cite[Appendix~A.5]{Kel98}. 
  The set $X$ is obtained from  $[0,1]$ by doubling the points in $Y$. Thus, every $y$ in $Y$ is replaced by two points $y^-$ and $y^+$ and one declares that $y^-<y^+$. This endows $X$ with a total order. The order topology on $X$ is the topology generated by open order intervals and the intervals of the form $[0,b)$ and $(a,1]$ for all $a$ and $b$ in $X$.
 Denote by $\pi\: X\to [0,1]$ the projection. Observe that $\pi$ is continuous since the preimages of every open interval in $[0,1]$ is an open order interval in $X$. This proves the first statement of the lemma and item 1.
 
Observe that the total order of $X$  has the least upper bound property, so the order topology on $X$ is compact; see, for instance, 
\cite[Theorem~27.1]{Mun00}. Also, observe that the order topology is Hausdorff, which, together with the compactness, implies that $X$ is regular. Since it is second countable, by the Urysohn Metrization Theorem, see, for instance, \cite[Theorem~34.1]{Mun00}, the order topology on $X$ is metrizable, proving item 2.

The map $\tmpf$ coincides with $\mpf$ on $[0,1]\setminus Y$ and  on $y^-$ and 
$y^+$ is defined by continuity. Thus~$\tmpf(x_1^-)=1$ and~$\tmpf(x_1^+)=0$, and for every~$y$ in~$Y$ different from~$x_1$ 
we have~$\tmpf(y^+)=\mpf(y)^+$ and $\tmpf(y^-)=\mpf(y)^-$. 
Since $\tmpf\: [0,x_1^-]\to \tmpf([0,x_1^-])$ and $\tmpf\: [x_1^+,1]\to \tmpf([x_1^+,1])$ are increasing bijections, $\tmpf$ is continuous.
From the definition of $X$ and $\tmpf$, we have that $\pi \circ \tmpf = \mpf \circ \pi$ \textcolor{black}{on ${X \setminus \{ x_1^+ \}}$,} which finishes the proof of item 3.

To show item 4, we first define a distance on $X$ compatible with the topology.  
Define an atomic measure $\lambda$ by 
\begin{equation}
\lambda \= \sum_{n=0}^{+\infty} \sum_{y\in f^{-n}(x_1)} 4^{-n}\delta_y,
\end{equation}
and the increasing maps $\iota^-, \iota^+\: [0,1]\to \R$ by 
\begin{equation}
\iota^-(x)\= x +\lambda([0,x)) \text{ and } \iota^+(x)\= x +\lambda([0,x]).
\end{equation}
Observe that 
\begin{align*}
& \iota^+(x) <  \iota^-(x'), \text{ for } x < x',\\
& \iota^-(x) =  \iota^+(x), \text{ for } x  \in [0,1]\setminus Y,\\
& \iota^+(y) =  \iota^-(y) + 4^{-n}, \text{ for } y \in  f^{-n}(x_1).
\end{align*}
We define $\iota\:X\to \R$ by 
\begin{equation}
\iota(x) = 
\begin{cases}
\iota^+(x), & \text{ if } x \in  X \setminus Y^-,\\
\iota^-(x), & \text{ if } x \in  Y^-.
\end{cases}
\end{equation}
Observe that  $\iota(X)$  is closed in  $\R$, so its order topology coincides with the induced topology from $\R$. Furthermore,  
$\iota \: X \to \iota(X)$  is an increasing bijection and thus a homeomorphism.
For all $x$ and $x'$ in $X$ define the distance
\begin{equation}
d(x,x') \= 
|\iota(x)-\iota(x')|.
\end{equation}
Put $\tY \= \pi^{-1}(Y)$. Now, we prove that the map $\tpi$ in item 4 is a conjugacy.
For every $n$ in $\N$, and every word $w$ on the alphabet $\{0,1\}$ of length $n$ denote by $[w]$ the cylinder in 
$\{0,1\}^{\N_0}$ that has the word $w$ in the first $n$ coordinates. We have that $\tpi^{-1}([w])$ is an interval $\tJ$ of the form $[a^+,b^-], [0,b^-]$ 
or $[a^+,1]$ with $a$ and $b$ in $Y$.  Since $[a^+,b^-] = (a^-,b^+), [0,b^-] = [0,b^+)$ 
and $[a^+,1]=(a^-,1]$ we get that 
$\tpi$ is continuous.  By Lemma \ref{l:inverse branches},  each of the intervals
 $(a,b], (0,b]$ or $(a,1]$ is the image  of $(0,1]$ by an inverse branch of $f^n|_{(0,1]}$, and denote by $J$ any of these intervals. 
 Since there is no other point of 
$\pi^{-1}(\bigcup_{k=0}^{n-1}f^{-k}(x_1))$ in $\tJ$ we have that
\begin{equation}
\text{diam}(\tJ) \le \text{diam}(J)+2^{-n}.
\end{equation}
By Lemma \ref{l:generating partition}, we have that $\text{diam}(J)\to 0$ as the length of $w$ goes to $+\infty$. Then, 
for every~$\omega$ in $\{0,1\}^{\N_0}$ we have that there is a unique~$x$ in~$X$ such that~$\tpi(x)=\omega$. Since $X$ is compact, we get that $\tpi$ is a homeomorphism, and since from the definition of  $\tpi$ we have $\tpi\circ \tmpf = \sigma \circ \tpi$, we conclude that the map~$\tpi$ is a conjugacy. 

Observe that a subset $B$ of $X$ is the same as the set $\pi(B)$ in $[0,1]$ with the points in $\pi(B)\cap Y$ duplicated. Then the \textsc{Borel} $\sigma$-algebra of $X\setminus \tY$ is the same as the \textsc{Borel} $\sigma$-algebra of $[0,1]\setminus Y$.
Since the set $Y$ does not support any invariant \textsc{Borel} probability for $\mpf$, the  same holds for  
$\tY$ and $\tmpf$. We deduce that for every invariant \textsc{Borel} probability measure~$\mu$ for~$\tmpf$ we have that~$\pi_*\mu$  is equal to the extension to the \textsc{Borel}~$\sigma$-algebra of~$[0,1]$ of the restriction of~$\mu$ to 
the \textsc{Borel}~$\sigma$-algebra of~$X\setminus \tY$, which coincides with the \textsc{Borel}~$\sigma$-algebra of~$[0,1]\setminus Y$.  On the other hand, for every invariant \textsc{Borel} probability measure $\nu$ for $\mpf$, there is a unique $\mu$ in $\sM_{\tmpf}$ such that $\pi_* \mu = \nu$. Therefore, the dynamical systems $(X,\tmpf,\mu)$ and $([0,1],\mpf,\nu)$ are measurably isomorphic. Furthermore, the map ${\pi_*\colon \sM_{\tmpf}\to \sM}$ is one-to-one. Since~$\sM_{\tmpf}$ is compact and~$\pi$ is continuous, it follows that~$\pi_*$ is a homeomorphism. Together with the fact that the measure-theoretic entropy map of~$\tmpf$ is upper-semicontinuous, see, for example, \cite[Proposition~2.19]{Bow08}, this implies the same property for~$f$. This finishes the proof of item 5 and the proof of the lemma.
\endproof

\subsection{Topological and tree pressures}
\label{ss:pressure formulas}
In this subsection, for every continuous potential $\varphi$, we give topological versions of the pressure $P(\varphi)$ in Proposition \ref{l:top pressure}. The proof of this proposition relies on the continuous extension in Lemma \ref{l:extension}.
\begin{prop}
\label{l:top pressure}
For every continuous potential~$\varphi\:[0,1]\to \R$, the following properties hold. 
\begin{itemize}
\item[1.] We have
\begin{equation}
\label{e:top pressure}
 \Pf(\varphi)= \lim_{n\to +\infty} \frac{1}{n} \log \sum_{Q\in \bigvee_{k=0}^{n-1} \mpf^{-k}\cP } \sup_{x\in Q}\exp(S_n \varphi (x)).
\end{equation}
In particular, $P(0)= \log 2$.
\item[2.] For every $y\in (0,1]$ we have 
\begin{equation}
\label{e:top pressure = tree pressure}
\Pf(\varphi)= \lim_{n\to +\infty} \frac{1}{n} \log \sum_{x\in \mpf^{-n}(y)} \exp(S_n \varphi (x)).
\end{equation}
\end{itemize}
\end{prop}
\proof Let $X,\pi, \tmpf, Y, Y^-$ and $Y^+$ be as in Lemma \ref{l:extension} and put~$\tvarphi \= \varphi \circ \pi$.

\partn{1} 
From Lemma~\ref{l:extension}(5)  
for every continuous  potential~$\varphi\:[0,1]\to \R$ and~$\tvarphi \= \varphi \circ \pi$
we have  
\begin{equation}
\label{e:equal pressure}
P_{\tmpf}(\tvarphi) = \sup_{\mu \in \sM_{\tmpf} } h_\mu(\tmpf)+ 
\int \tvarphi \dd \mu
=  
\sup_{\nu \in \sM } h_\nu + \int \varphi \dd \nu
= \Pf(\varphi).
\end{equation}
Let $\tP_0$ and $\tP_1$ be the sets in Lemma \ref{l:extension}(4). Observe that  $\sC=\{\tP_0,\tP_1\}$ is the open cover of $X$ corresponding to the open cover $\{[0],[1]\}$ in $\{0,1\}^{\N_0}$. Thus, by the Variational Principle (see for instance \cite[2.17. Variational Principle and Lemma~1.20]{Bow08}), we get that 
\begin{equation}
\label{e:vp}
P_{\tmpf}(\tvarphi) = 
\lim_{n\to +\infty} \frac{1}{n} \log \sum_{P\in \bigvee_{k=0}^{n-1} \tmpf^{-k}\sC } \sup_{x\in P}\exp(S_n \tvarphi (x)).
\end{equation}
Notice that by the equation $\pi\circ\widetilde f=f\circ\pi$ on
$X\setminus\{x_1^+\}$ in Lemma~\ref{l:extension}(3), for every
$n\in\mathbb N$ we have
\[
S_n\widetilde\varphi(x)=S_n\varphi(\pi(x))
\]
whenever $\widetilde f^k(x)\neq x_1^+$ for every
$k\in\{0,\ldots,n-2\}$. Also observe that every
$Q\in\bigvee_{k=0}^{n-1}f^{-k}\mathcal P$ is of the form
$[0,b]$, $(a,b]$, or $(a,1]$, and that the corresponding element of
$\bigvee_{k=0}^{n-1}\widetilde f^{-k}\widetilde{\mathcal C}$ is of the
form $[0,b^-]$, $[a^+,b^-]$, or $[a^+,1]$, respectively. In
particular, no point of $(a^+,b^-]$ visits $x_1^+$ during its first
$n-1$ iterates.
Together with  the continuity of $\tvarphi$  and $\tmpf$ we get
$$
\sup_{x\in [a^+,b^-]}\exp(S_n \tvarphi (x)) =  \sup_{x\in (a^+,b^-]}\exp(S_n \tvarphi (x)) = \sup_{z\in (a,b]}\exp(S_n \varphi (z)),
$$
and the same holds for $Q$ of the form $[0,b]$ or $(a,1]$. Then,
\begin{equation}
  \lim_{n\to +\infty} \frac{1}{n} \log \sum_{P\in \bigvee_{k=0}^{n-1} \tmpf^{-k}\sC } \sup_{x\in P}\exp(S_n \tvarphi (x))
  = 
\lim_{n\to +\infty} \frac{1}{n} \log \sum_{Q\in \bigvee_{k=0}^{n-1} \mpf^{-k}\cP} \sup_{x\in Q}\exp(S_n \varphi (x)).
\end{equation}
Together with \eqref{e:equal pressure} and \eqref{e:vp}, this implies \eqref{e:top pressure}. Also from  \eqref{e:equal pressure} and Lemma \ref{l:extension}(4), we get that 
$P(0)=P_{\tmpf}(0)=\log 2$. 

\partn{2} Fix $y$ in $(0,1]$. By  Lemma \ref{l:generating partition}, the diameter of the elements in $ \bigvee_{k=0}^{n-1} \mpf^{-k}\cP$ tends  to zero as $n$ goes to $+\infty$. Then,  since $\varphi$ is uniformly continuous on $[0,1]$, for every $\varepsilon>0$ there is $n_0\in \N$ such that for every integer $n\ge n_0$, every $Q\in \bigvee_{k=0}^{n-1} \mpf^{-k}\cP$ and 
every $x_Q\in Q\cap \mpf^{-n}(\{y\})$ we have 
$$
\sup_{z\in Q} |S_{n-n_0}\varphi(z) - S_{n-n_0} \varphi(x_Q)| \le (n-n_0)\varepsilon.
$$
Then, 
\begin{equation}
\label{e:tree and extremes 1}
\exp(S_n \varphi(x_Q)) \le \sup_{z\in Q}\exp(S_n \varphi (z))  
\le \exp((n-n_0)\varepsilon+ n_0 \|\varphi\| ) \exp(S_n \varphi(x_Q)),
\end{equation}
which implies 
\begin{equation}
\begin{split}
\label{e:tree and extremes }
\sum_{x\in \mpf^{-n}(y)} \exp(S_n \varphi (x)) \le & \sum_{Q\in \bigvee_{k=0}^{n-1} \mpf^{-k}\cP} \sup_{x\in Q}\exp(S_n \varphi (x))  \\ 
\le &  \exp((n-n_0)\varepsilon+ n_0 \|\varphi\|) \sum_{x\in \mpf^{-n}(y)} \exp(S_n \varphi (x)).
\end{split}
\end{equation}
Therefore, by item 1, we have
\begin{equation}
\limsup_{n\to +\infty} \frac{1}{n} \log \sum_{x\in \mpf^{-n}(y)} \exp(S_n \varphi (x))
\le 
\Pf(\varphi) 
\le 
\varepsilon + \liminf_{n\to +\infty} \frac{1}{n} \log \sum_{x\in \mpf^{-n}(y)} \exp(S_n \varphi (x)),
\end{equation}
which  implies \eqref{e:top pressure = tree pressure}.
\endproof

\subsection{Dichotomy for \textsc{Hölder} continuous potentials}
\label{ss:dichotomy}

\textcolor{black}{
  The rest of the paper relies several times on the following dichotomy.
  }

\begin{prop}[Dichotomy for \textsc{Hölder} continuous potentials]
  \label{c:tpt}
  For every \textsc{Hölder} continuous potential~$\varphi$, the following dichotomy holds:
  \begin{enumerate}
  \item[1.] For every $\beta \in (0,+\infty)$, we have
    $P(\beta\varphi) > \beta\varphi(0)$, and the function
    $\beta \mapsto P(\beta\varphi)$ is real-analytic on $(0,+\infty)$; or
  \item[2.] There exists $\beta_0 > 0$ such that
    $P(\beta_0\varphi) = \beta_0\varphi(0)$.
    Define
    \begin{equation}
      \label{eq:beta-star-tpt}
      \beta_* \= \inf\left\{ \beta > 0 \colon P(\beta\varphi) \le \beta\varphi(0) \right\}.
    \end{equation}
    Then $\beta_* > 0$; for every $\beta \in (0,\beta_*)$, we have
    $P(\beta\varphi) > \beta\varphi(0)$, and for every $\beta \ge \beta_*$, we have
    $P(\beta\varphi) = \beta\varphi(0)$. Furthermore, the function
    $\beta \mapsto P(\beta\varphi)$ is real-analytic on
    $(0,+\infty)\setminus\{\beta_*\}$.
  \end{enumerate}
\end{prop}

\textcolor{black}{
  The following corollary is a direct consequence of Proposition~\ref{c:tpt} and the Key Lemma in~\S\ref{ss:about the proof}.
}
\textcolor{black}{
  \begin{coro}
  \label{c:ergdodic optimization}
    If a \textsc{Hölder} continuous potential~$\varphi$ undergoes a phase transition in temperature, then~$\delta_0$ is the unique maximizing measure for~$\varphi$.
  \end{coro}
}
\textcolor{black}{
  The proof of Proposition~\ref{c:tpt} relies on the following consequence of \cite[Theorem~A.2(2)]{InoRiv26}: If~$\psi$ is a H{\"o}lder continuous potential satisfying ${P(\psi) > \psi(0)}$, then the pressure is real-analytic on a neighborhood of~$\psi$ in the H{\"o}lder topology.
  Indeed, the hypothesis ${P(\psi) > \psi(0)}$ implies that the potential~$\psi$ is hyperbolic for~$f$ in the sense of \cite{InoRiv12}.
  This follows from \cite[Proposition~A.8]{InoRiv26}, which is an adaptation to intermittent maps of \cite[Proposition~3.1]{InoRiv12}.
  Combined with the existence of an atom-free $\exp(\psi - S_n\psi)$-conformal measure \cite[Lemma~A.6]{InoRiv26}, the hyperbolicity of~$\psi$ implies the spectral gap property for the associated transfer operator by \cite[Theorem~1 in~\S4.2]{LiRiv14b}, which reformulates results of \textsc{Keller} \cite[Theorems~3.2 and~3.3]{Kel85}.
  The desired analyticity statement then follows from the spectral gap property.
}

\proof[Proof of Proposition~\ref{c:tpt}]
By \cite[Theorem~A.2(2)]{InoRiv26}, if~$\psi$ is a Hölder continuous
potential satisfying~$P(\psi)>\psi(0)$, then the pressure is real-analytic
on a neighborhood of~$\psi$ in the Hölder topology. In particular, for every
Hölder continuous potential~$\eta$, the function~$t\mapsto P(\psi+t\eta)$
is real-analytic at~$0$. Applying this with~$\psi=\beta\varphi$ and
$\eta=\varphi$, we obtain that if~$P(\beta\varphi)>\beta\varphi(0)$ for
some~$\beta>0$, then the map~$\beta'\mapsto P(\beta'\varphi)$ is
real-analytic on a neighborhood of~$\beta$.
Assume now that there exists $\beta_0 \in (0,+\infty)$ such that
$P(\beta_0\varphi)=\beta_0\varphi(0)$. Then, for every $\beta \ge \beta_0$, by
definition of pressure we have
\[
\frac{1}{\beta}P(\beta\varphi)\le \frac{1}{\beta_0}P(\beta_0\varphi),
\]
and therefore
\[
\beta\varphi(0)\le P(\beta\varphi)\le
\frac{\beta}{\beta_0}P(\beta_0\varphi)=\beta\varphi(0).
\]
Thus $P(\beta\varphi)=\beta\varphi(0)$ for every $\beta \ge \beta_0$.

Since $P(\beta\varphi)\ge \beta\varphi(0)$ for every $\beta>0$, the set
\[
\left\{ \beta > 0 \colon P(\beta\varphi)\le \beta\varphi(0)\right\}
\]
coincides with
\[
\left\{ \beta > 0 \colon P(\beta\varphi)= \beta\varphi(0)\right\}.
\]
Hence, if alternative~(2) holds, the previous paragraph implies that this set is
of the form $[\beta_*,+\infty)$, so the stated description of~$\beta_*$ follows.

Finally, since the pressure is continuous with respect to the potential and
$P(0)=\log 2>0$ (see Proposition~\ref{l:top pressure}), we have $\beta_*>0$.
This completes the proof.
\endproof

\section{The Key Lemma}
\label{s:key lemma}
In this section, we prove the Key Lemma stated in \S\ref{ss:about the proof}.
\subsection{Iterated function system and Proposition \ref{p:good IFS}}
Recall that we have fixed $\alpha$ in $(0,+\infty)$, that $f$ denotes the \textsc{Manneville–Pomeau} map of parameter $\alpha$ defined in \S\ref{ss:pt}, that $\sM$ denotes the space of \textsc{Borel} probability measures invariant by $f$, that $x_1$ is the discontinuity point of $f$ and that $J_0$ is the interval $(x_1,1]$.

Let $\mu$ be in $\sM$ different from $\delta_0$.
Denote by $\supp(\mu)$ the topological support of $\mu$. When $0\notin \supp(\mu)$ one can use that for \textsc{Hölder} continuous potentials on topologically mixing subshift of finite type the equilibrium states have positive entropy (see for instance \cite[Theorem~1.25]{Bow08}) to prove that $P(\varphi) > \int \varphi \dd \mu$. For the general case, we use an iterated function system.

By the ergodic decomposition theorem, we can assume that $\mu$ is an ergodic measure.  
Let $\Sigma$ be the set of all infinite words in the alphabet $\N$ and for every $n$ in $\N$ put 
\begin{equation}
\Sigma_n \= \N^n
\text{ and }
\Sigma^* \=  \bigcup_{n\in \N} \N^n.
\end{equation}
For every $n$ in $\N$ and every $\uell$ in $\Sigma_n$  the \emph{length of} $\uell$ is $n$ and it is denoted by $|\uell|$.
An infinite sequence of pairwise distinct functions  $(\phi_\ell)_{\ell \in \N}$ from $(x_1,1]$ into $(x_1,1]$ is called an \emph{Iterated Function System (IFS)}. For every $n$ in $\N$ and every finite word $\ell_1\cdots \ell_n$ in $\Sigma_n$
 put 
 \begin{equation}
\phi_{\ell_1\cdots \ell_n} \= \phi_{\ell_1}\circ \cdots \circ \phi_{\ell_n}.
\end{equation}
We say that the IFS is \emph{free}  if for all $\uell$ and $\uell'$ in $\Sigma^*$ with $\uell\neq \uell'$ we have that $\phi_{\uell}$ is different from 
$\phi_{\uell'}$. 
We say that the IFS is \emph{generated by $\mpf$} 
 if  for every $\ell$ in $\N$ there is $m_\ell$ in $\N$ such that  $\mpf^{m_\ell} \circ \phi_\ell$ is the identity on $(x_1,1]$.
 We say that $(m_\ell)_{\ell\in \N}$ is  the \emph{time sequence} of $(\phi_{\ell})_{\ell\in \N}$.
 For every $n$ in $\N$ and every finite word $\ell_1\cdots \ell_n$ in $\Sigma_n$
 put 
 \begin{equation}
m_{\ell_1\cdots \ell_n} \= m_{\ell_1} + \cdots + m_{\ell_n}.
\end{equation}
We say that $(\phi_\ell)_{\ell \in \N}$ is \emph{hyperbolic with respect to $\mpf$}, if there are constants $C>0$ and $\lambda>1$ such that for every $x\in (x_1,1]$, for every $\uell\in \Sigma^*$  and for every $j\in \{1, \dots, m_{\uell}\}$ one has
\begin{equation}
\label{e:hyp IFS}
|D\mpf^j(\mpf^{m_{\uell}-j}(\phi_{\uell}(x)))|\ge C\lambda^j.
\end{equation}

\begin{prop}
\label{p:good IFS}
For each \textsc{Hölder} continuous potential~$\varphi$ on~$[0,1]$ and for each  ergodic measure~$\mu$ in $\sM$  
distinct from~$\delta_0$, the following holds. There are a constant $C>0$ and a free hyperbolic IFS 
$(\phi_\ell)_{\ell\in \N}$  generated by $\mpf$ with strictly increasing time sequence $(m_\ell)_{\ell \in \N}$ such that
\begin{equation}
\label{e:good time}
\inf_{z\in (x_1,1]}S_{m_\ell} \varphi (\phi_\ell(z)) \ge m_\ell \int \varphi \dd \mu - C.
\end{equation}
\end{prop}
The proof of Proposition \ref{p:good IFS} is in \S\ref{ss:pp}. Now, we assume the result and prove the remaining Key Lemma.

\subsection{Proof of the Key Lemma assuming Proposition \ref{p:good IFS}}
Assume that $\mu$ is an ergodic distinct from $\delta_0$. Let $(\phi_\ell)_{\ell\in \N}$ be the IFS given by Proposition \ref{p:good IFS} with time sequence $(m_\ell)_{\ell \in \N}$ and constant $C$. Fix $z_0$ in $(x_1,1]$. For every $N$ in $\N$ put
\begin{equation}
\Lambda_N \= \sum_{\uell \in \Sigma^*, m_{\uell}=N} \exp(S_N \varphi (\phi_{\uell}(z_0))).
\end{equation} 
Since the IFS $(\phi_\ell)_{\ell\in \N}$ is free,  by \eqref{e:top pressure} we have  
\begin{equation}
\Pf(\varphi) \ge  \limsup_{N\to +\infty} \frac{1}{N} \log \Lambda_N.
\end{equation} 
Now consider the following generating function
\begin{equation}
\label{e:g function 1}
\Xi(s) \= \sum_{N\in \N } \Lambda_N s^N.
\end{equation}
The radius of convergence $R$ of $\Xi(s)$ is at least $\exp(-\Pf(\varphi))$. 
Observe  that
\begin{equation}
\label{e:g function 2}
\Xi(s) = \sum_{ \uell \in \Sigma^* }  \exp(S_{m_{\uell}} \varphi (\phi_{\uell}(z_0))) s^{m_{\uell}}.
\end{equation} 
By \eqref{e:good time}  in Proposition \ref{p:good IFS} for every $\uell$ in $\Sigma^*$ we have that
\begin{equation}
S_{m_{\uell}}\varphi(\phi_{\uell}(z_0)) \ge m_{\uell} \int \varphi \dd \mu - |\uell| C.
\end{equation}
Then, defining 
\begin{equation}
\Phi(s)\= \sum_{\ell = 1}^{+\infty} \exp\left(m_{\ell} \int \varphi \dd \mu -  C\right) s^{m_{\ell}}
\end{equation}
we get that the power series in $s$
\begin{equation}
\label{e: powers of Phi}
\Phi(s) + \Phi(s)^2 + \Phi(s)^3 + \cdots
\end{equation}
has coefficients smaller than or equal to the corresponding coefficients of~$\Xi(s)$. 
Observe that, since $(m_\ell)_{\ell \in \N}$ is strictly increasing, the radius of convergence of~$\Phi(s)$
is~$\hR =\exp(-\int \varphi \dd \mu)$ and 
that
\begin{equation}
\lim_{s\to \hR^-} \Phi(s) = +\infty.
\end{equation}
Then, there is $s_0$  in $(0,\hR)$ such that $\Phi(s_0)$ is finite and  $\Phi(s_0)\ge 1$, and this implies that the radius of convergence of the series  \eqref{e: powers of Phi} is strictly smaller than $s_0$. Then
\begin{equation}
\exp(-P(\varphi)) \le R < s_0 < \hR = \exp\left(-\int \varphi \dd \mu\right),
\end{equation}
finishing the proof of the lemma.

\subsection{Proof of Proposition \ref{p:good IFS}}
\label{ss:pp}
Before proving Proposition \ref{p:good IFS}, we demonstrate two lemmas that establish that an IFS is free and has bounded distortion, along with a well-known but folklore result in abstract Ergodic Theory, for which we provide a proof for the reader's convenience.
\begin{lemm}
\label{l:free}
Let $(z_n)_{n\in \N_0}$ be a sequence in $[0,1]$ such that $z_0$ is in $J_0$ and  for every $n$ in $\N_0$ we have that $z_{n}=\mpf(z_{n+1})$.
Let $M\ge 1$ be an integer and let $(n_\ell)_{\ell\in \N}$ be a strictly increasing sequence of positive integers such that 
$n_{\ell+1} \ge n_{\ell}+M$. For every $\ell$ in $\N$, let  $x_\ell$ be a point of $J_0$ in 
$\mpf^{-M}(z_{n_\ell})$ different from $z_{n_\ell+M}$,  and let  $\phi_\ell$ be the inverse branch of $\mpf^{n_\ell+M}$ from $J_0$ into itself such that $\phi_\ell(z_0)=x_\ell$. Then, the IFS $(\phi_\ell)_{\ell\in \N}$ generated by $f$ is free with time sequence 
$(n_\ell+M)_{\ell\in \N}$.
\end{lemm}
\proof
Let $\uell = \ell_1\cdots\ell_n$ and $\uell' = \ell'_1\cdots \ell'_{k}$ be in $\Sigma^*$ with $\uell\neq \uell'$. 
Assume that  $m_{\uell} \neq m_{\uell'}$. Without loss of generality we assume that  
$m_{\uell} < m_{\uell'}$. Suppose we had $\phi_{\uell} = \phi_{\uell'}$
then  $f^{m_{\uell'}}\circ \phi_{\uell}  = \text{Id}|_{J_0}$, which implies that $f^{m_{\uell'}-m_{\uell}}= \text{Id}|_{J_0}$ and thus, 
$m_{\uell} = m_{\uell'}$ giving a contradiction. 
Now assume that $m_{\uell} = m_{\uell'}$. 
We also assume that $\ell_n\neq \ell'_k$, the general case, can be reduced to this one.  Without loss of generality, we assume that  
$\ell_{n} < \ell'_{k}$. In particular, we have $m_{\ell_n} < m_{\ell'_k}$. Suppose we had
$\phi_{\uell} = \phi_{\uell'}$ then $f^{m_{\uell} - m_{\ell_n}} \circ \phi_{\uell} = f^{m_{\uell'} - m_{\ell_n}} \circ  \phi_{\uell'}$ and thus, 
$$ 
\phi_{\ell_n} = f^{m_{\ell'_k} - m_{\ell_n}} \circ  \phi_{\ell'_k} = f^{n_{\ell'_k} - n_{\ell_n}} \circ  \phi_{\ell'_k}.
$$ 
Evaluating this last equality at $z_0$ and using that $n_{\ell'_k} - n_{\ell_n}\ge M$ we get that
$$
x_{\ell_n} = f^{n_{\ell'_k} - n_{\ell_n}}(x_{\ell'_k}) = f^{n_{\ell'_k} - n_{\ell_n}-M}(z_{n_{\ell'_k}}) = z_{n_{\ell_n}+M}. 
$$
However, by hypothesis, these two points are different.  Therefore,  $\phi_{\uell} \neq \phi_{\uell'}$ and this concludes the proof of the lemma.
\endproof
For every $\gamma$ in $(0,1]$ denote by $C^\gamma(\R)$ the space of $\gamma$-\textsc{Hölder} continuous real-valued functions. For every $\varphi$ in $C^\gamma(\R)$, put 
\begin{equation}
|\varphi|_\gamma \= \sup_{x,y \in [0,1], x\neq y} \frac{|\varphi(x)-\varphi(y)|}{|x-y|^\gamma}.
\end{equation}
\begin{lemm}
\label{l:IFS distortion}
For each \textsc{Hölder} continuous potential $\varphi\:[0,1]\to \R$ and for each
IFS  $(\phi_\ell)_{\ell\in \N}$  generated by $f$,  hyperbolic  with respect to $\mpf$ and with time sequence $(m_\ell)_{\ell \in \N}$
the following holds. There is a constant 
$\Delta'>0$ such that for every $\uell $ in $\Sigma^*$ and  all $x$ and $y$ in $J_0$ we have that 
\begin{equation}
|S_{m_{\uell}} \varphi (\phi_{\uell}(x))-S_{m_{\uell}} \varphi (\phi_{\uell}(y))|\le \Delta'.
\end{equation}
\end{lemm}
\proof
Let $\gamma $ be in $(0,1]$ such that  potential $\varphi$ is in $C^\gamma(\R)$.
By \eqref{e:hyp IFS} there are  constants $C>0$ and $\lambda>1$ such that   for every $j$ in $\{1,\ldots,m_{\uell}\}$  we have that 
$$
|f^{m_{\uell}-j}(\phi_{\uell}(x))-f^{m_{\uell}-j}(\phi_{\uell}(y))| \le C\lambda^{-j}.
$$
Then,
\begin{equation}
|S_{m_{\uell}} \varphi (\phi_{\uell}(x))-S_{m_{\uell}} \varphi (\phi_{\uell}(y))|\le |\varphi|_{\gamma}C^{\gamma} \sum_{j=1}^{m_{\uell}}(\lambda^{\gamma})^{-j}.
\end{equation}
Taking $\Delta' \= |\varphi|_{\gamma}C^{\gamma} \sum_{j=1}^{+\infty}(\lambda^{\gamma})^{-j}$ we finish the proof of the lemma.
\endproof

\begin{lemm}
\label{l:nonnegative sums}
Let~$(X,\sB,\mu)$ be a probability space and let~$T\:X\to X$ be an ergodic measure-preserving map. 
Then, for each integrable function~$\varphi\:X\to \R$  with null integral, there is a full measure set of~$x$ in $X$ such that
\begin{equation}
\limsup_{n\to+\infty} S_n\varphi(x)\ge 0.
\end{equation}
\end{lemm}
\proof
It is enough to prove that for every $\varepsilon>0$ and every $k$ in $\N$ the set
\begin{equation}
A\=\{x\in X \colon \text{ for every } n\in \N \text{ such that } n \ge k \text{ one has } S_n\varphi \le -\varepsilon\}
\end{equation}
has measure 0. Suppose we had that there are $\varepsilon>0$ and $k$ in $\N$ such that $\mu(A)>0$. By the 
Birkhoff Ergodic Theorem there is $x$ in $A$ such that
\begin{equation}
\label{e:birkhoff 0}
\lim_{n\to +\infty} \frac{1}{n} S_n\varphi(x) = 0,
\end{equation}
and 
\begin{equation}
\label{e:birkhoff positive}
\lim_{n\to +\infty} \frac{1}{n} S_n\textbf{1}_A(x) = \mu(A).
\end{equation}
Denote by $(n_\ell)_{\ell\in \N_0}$ the sequence of return times of $x$ to $A$ with $n_0=0$. 
For every $\ell$ in $\N_0$ we have $n_\ell \ge \ell$ and 
since $\mu(A)>0$ we also have that 
there is $P$ in $[1,+\infty)$ such that for every $\ell$ in $\N_0$,
\begin{equation}\label{e:lineal returns}
n_\ell \le P \ell.
\end{equation}
Since for every $m$ in $\N$ we have $n_{(m+1)k}\ge n_{mk}+k$, by \eqref{e:lineal returns}, for every $m$ in $\N$ we get
\begin{equation}
S_{n_{mk}} \varphi(x) = \sum_{j=0}^{m-1} S_{n_{(j+1)k}-n_{jk}}\varphi(T^{n_{jk}}x) \le m(-\varepsilon) \le \frac{n_{mk}}{Pk}(-\varepsilon).
\end{equation}
Together with \eqref{e:birkhoff 0} this implies 
\begin{equation}
0 = \lim_{m\to +\infty} \frac{1}{n_{mk}}S_{n_{mk}} \varphi(x)  \le  \frac{1}{Pk}(-\varepsilon) < 0,
\end{equation}
which gives a contradiction and finishes the proof of the lemma.
\endproof

Put $I\=[0,1]$ and let $\mu$ be an ergodic measure in $\sM$ distinct from $\delta_0$. Denote by $(\hI,\hmpf)$ the natural extension  of $(I,\mpf)$. 
That is, $\hI$ is the set of sequences $(z_n)_{n\in \N_0}$ in $I$ such that for every  $n$ in $\N_0$ one has $z_n = \mpf(z_{n+1})$, and 
$\hmpf$ is the bijective map  from $\hI $ onto  $\hI$ defined for every $(z_n)_{n\in \N_0}$ in $\hI$ by 
\begin{equation}
\hmpf((z_n)_{n\in \N_0}) = (\mpf(z_0), z_0, z_1, z_2,\ldots, z_n, \ldots).
\end{equation}
The space $\hI$ inherits a \textsc{Borel} $\sigma$-algebra as a subset of the product space $\Pi_{n=0}^{+\infty} I$ 
and the map $\hf$ is a measurable isomorphism.  Denote by $\Pi\:\hI\to I$ the projection onto the zeroth coordinate. We have that 
$\Pi\circ \hmpf = \mpf\circ \Pi$, and that there is a unique invariant probability measure $\nu$ for $\hmpf$ such that 
$\Pi_* \nu = \mu$. 
Since $\mu$ is ergodic, the measure $\nu$ is also ergodic for~$\hmpf$ and~$\hmpf^{-1}$. 
Observe that since $\mu$ is  different from $\delta_0$ we have $\mu(J_0)>0$ and  $\nu(\Pi^{-1}(J_0))>0$. Put 
$\hI' \=\Pi^{-1}(J_0)$. Fix $\uz$ in $\hI'$ such that $\uz$ is  generic for  the Ergodic Theorem for $\hmpf^{-1}$, $\nu$ and the bounded measurable function  $\log D\mpf \circ  \Pi$. Then, we have 
\begin{equation}
\label{e:positive Lyapunov}
\lim_{n\to +\infty}\frac{1}{n}\sum_{j=0}^{n-1} \log D\mpf \circ  \Pi (\hmpf^{-j}\uz) = \chi_\mu(\mpf) > 0.
\end{equation}
By Lemma \ref{l:nonnegative sums}, we can choose $\uz$ so that, in addition, we have
\begin{equation}
\limsup_{n\to +\infty} \sum_{j=0}^{n-1} \left( \varphi \circ  \Pi (\hmpf^{-j}\uz) - \int \varphi \dd \mu \right) \ge 0.
\end{equation}
By Lemma \ref{l:inverse branches}  for every  $\ell$ in $\N$ there is a unique inverse branch $\tphi_\ell$  of $\mpf^\ell$ defined on $(0,1]$ verifying
$\Pi(\hmpf^{-\ell}\uz)\in \tphi_\ell((0,1])$. 
By Lemma \ref{l:geometric-distortion} and  \eqref{e:positive Lyapunov} there are $C'>0$ and $\lambda>1$ such that for every $n$ in $\N$ and every 
$z$ in $(x_1,1]$ we have
\begin{equation}
\label{e:expanding IFS}
D\mpf^n(\tphi_n(z)) \ge C'\lambda^n.
\end{equation}
Then, there is a strictly increasing sequence $(n_\ell)_{\ell\in \N}$ in $\N$ such that 
\begin{equation}
\label{e:sum and times}
S_{n_\ell}\varphi (z_{n_\ell}) \ge n_{\ell}\int \varphi \dd \mu - 1, n_{\ell+1} \ge n_\ell + 3, C'\lambda^\frac{n_1}{2} > \lambda^\frac{3}{2},
\end{equation} 
and  the sequence $(z_{n_\ell})_{\ell \in \N}$ converges to  some $w$ in $I$.

Now, we distinguish two cases. Let~$y_2$  be the unique point in~$(x_1,1]$ that satisfies~$x_1=\mpf(y_2)$,
and observe that $\mpf^2$ maps $(x_1,y_2]$ and $(y_2,1]$  bijectively onto $(0,1]$.
Let $y'$ and $y''$ be the preimage by $\mpf^2$ of $x_1$ in $(x_1,y_2]$ and $(y_2,1]$, respectively.
We have that $y'<y_2<y''$ and that $\mpf^3$ maps each of the intervals $(x_1,y']$ and $(y'',1]$  bijectively onto $(0,1]$. 
The first case is when $w\in [0,y_2]$.  
By passing to a subsequence, we can assume  that  
\begin{equation}
\label{e:free1}
(z_{n_\ell})_{\ell\in \N} \text{ is in }(0,y''].
\end{equation}
For every $\ell \in \N$ we put
\begin{equation}
\label{e:branch def1}
\phi_\ell \=  (\mpf^3|_{(y'',1]})^{-1}\circ \tphi_{n_\ell}|_{J_0}. 
\end{equation} 

The second case is when  $w\in (y_2,1]$. 
Again, by passing to a subsequence, we can assume that
\begin{equation}
\label{e:free3}
(z_{n_\ell})_{\ell\in \N} \text{ is in }(y',1].
\end{equation}
For every $\ell \in \N$ we put
\begin{equation}
\label{e:branch def3}
\phi_\ell \=  (\mpf^3|_{(x_1,y']})^{-1}\circ \tphi_{n_\ell}|_{J_0}. 
\end{equation} 
In both cases, put $m_\ell \= n_\ell +3$. We have that $(\phi_\ell)_{\ell \in \N}$ is an IFS generated by $f$ with time sequence 
$(m_\ell)_{\ell \in \N}$.

By \eqref{e:free1}, \eqref{e:branch def1}, \eqref{e:free3} and \eqref{e:branch def3}, we have in all of the cases  that  for every $\ell$ in $\N$, the point
$x_\ell \= \phi_\ell(z_0)$  is in $f^{-3}(z_{n_\ell})$  and it is  different from $z_{m_\ell}$. Thus, by Lemma \ref{l:free} and the second inequality in \eqref{e:sum and times}, we get that the IFS $(\phi_\ell)_{\ell \in \N}$ is free.

Now we prove that the IFS $(\phi_\ell)_{\ell \in \N}$ is hyperbolic.  By \eqref{e:expanding IFS}, the second and third inequalities in \eqref{e:sum and times},  for every $z$ in $J_0$ and for every $j\in \{0,1,2\}$ we have in all of the cases that 
\begin{equation}
D\mpf^{m_\ell-j}(f^j (\phi_{\ell}(z)))  \ge \lambda^{\frac{m_\ell-j}{2}}.
\end{equation}
Again together with  \eqref{e:expanding IFS} this implies that  for every $z$ in $J_0$,  for every $\uell\in \Sigma^*$  and for every $j\in \{1, \dots, m_{\uell}\}$ one has
\begin{equation}
\label{e:hyperbolicity}
D\mpf^j( f^{m_{\uell}-j}(\phi_{\uell}(z))) \ge C' \lambda^{\frac{j}{2}},
\end{equation}
and thus, the IFS $(\phi_\ell)_{\ell \in \N}$ is hyperbolic with constants $C'>0$ and $\lambda^{\frac{1}{2}}>1$.

It remains to prove \eqref{e:good time}. Put 
\begin{equation}
C_1\= -\inf_{x\in [0,1]} \varphi \text{ and } C''\= 1+ 3C_1 +3 \int \varphi \dd\mu.
\end{equation}
By the first inequality in  \eqref{e:sum and times}, we have 
\begin{equation}
\label{e:sum z_0}
\begin{split}
S_{m_\ell}\varphi(\phi_{\ell}(z_0)) & = S_{3}\varphi(\phi_{\ell}(z_0))+S_{n_\ell}\varphi(\tphi_{n_\ell}(z_0)) \\
& \ge -3C_1  + n_\ell \int \varphi \dd\mu -1 \\
& \ge   m_\ell \int \varphi \dd\mu -C''.
\end{split}
\end{equation}
Together with  Lemma \ref{l:IFS distortion}, this finishes the proof of   \eqref{e:good time}  with $C= C''+\Delta'$, concluding the proof of Proposition \ref{p:good IFS}.

\section{The Bowen-Type Formula and applications}
\label{s:inducing}

In this section, we prove the Bowen-Type Formula and the other results used in the proofs of Theorems~\ref{t:sarig}, \ref{t:sarig alpha}, and \ref{t:robust sarig} in \S\ref{s:main}, such as Proposition~\ref{l:operator} and Proposition~\ref{p:omega transitions}. We also derive additional consequences, including Corollary~\ref{c:phase transitions} and Proposition~\ref{ex:geometric}.

\subsection{Induced map and the two-variable pressure}
\label{sss: induced map} 

We introduce the inducing scheme and the two-variable pressure functions, which are the primary tools for proving the existence of phase transitions in temperature.

For every~$n\in \N$ let~$y_n$ be the unique point in~$(x_1,1]$ such that~$\mpf(y_n)=x_{n-1}$, and
put 
$$
I_n\=(y_{n+1},y_{n}].
$$ 
We have that for every $n$ in $\N$, the map  $\mpf^n$ sends  $J_n$ and $I_n$ diffeomorphically onto $(x_1,1]$.
Define $m\:(0,1]\to \N$ by 
$$
m^{-1}(n) \= I_n \cup J_n,
$$ 
$F\:J_0\to J_0$ by 
$$
F(x)\=\mpf^{m(x)}(x),
$$
and 
$L\:(0,x_1]\to J_0$ by 
$$
L(x)\=\mpf^{m(x)}(x).
$$
The maps $F$  and $L$ are called \emph{the first return map}  and \emph{the first landing map} of $\mpf$ onto $J_0$, respectively.  

Let $\fD$ be the partition of $J_0$ given by the intervals $I_n$ with $n$ in $\N$. 
For every continuous potential $\varphi\: [0,1]\to \R$, every $p$ in $\R$, and every $\ell$ in $\N$, we define
\begin{equation}
\label{d:Zell}
 Z_\ell (\varphi,p)
 \=  
\sum_{J\in \bigvee_{j=0}^{\ell-1} F^{-j}(\fD)} 
\exp\left( \sup_{x\in J} S_{m(x)+\cdots+m(F^{\ell-1}(x))} (\varphi-p)(x)\right).
\end{equation}
The sequence $ (Z_\ell (\varphi,p) )_{\ell \in \N}$ is in $(0,+\infty]$ and it is submultiplicative.
Thus
\begin{equation}\label{e:Zell}
\left(  \log Z_\ell (\varphi,p) \right)_{\ell \in \N}
\end{equation}
is a subadditive sequence in $\R\cup \{+\infty\}$.  Here, we use the convention that $\log(+\infty) = +\infty$. When the limit 
\begin{equation}\label{e:P defi}
\lim_{\ell\to +\infty} \frac{1}{\ell} \log Z_\ell (\varphi,p) 
\end{equation}
exists in the extended real line $\R\cup\{-\infty,+\infty\}$, we denote it by $\sPf(\varphi,p)$ and call it 
\emph{the two-variable pressure function for the potential $\varphi$ with parameter $p$}. 
It is exactly the pressure for the potential defined on $J_0$ by
\begin{equation}
S_{m(x)+\cdots+m(F^{\ell-1}(x))} \varphi(x) -(m(x)+\cdots+m(F^{\ell-1}(x))p
\end{equation}
of the induced system $(J_0,F)$ viewed as a full shift on countable many symbols \cite{UrbRoyMun22b}.

When the sequence~$\left(  \log Z_\ell (\varphi,p) \right)_{\ell \in \N}$ is finite, by the Subadditive Lemma, the limit~\eqref{e:P defi} exists and is in~$\R\cup \{-\infty\}$. In particular,~$Z_1(\varphi,p)< +\infty$ implies~$\sP(\varphi,p)<+\infty$. From Proposition~\ref{l:operator} below, we have that for potentials~$\varphi$ with bounded variation ergodic sums\\~$\sP(\varphi,p)<+\infty$ implies~$Z_1(\varphi,p)<+\infty$.

Now, we introduce a transfer operator for the induced system, which is useful for computing the two-variable pressure.
Denote by~$L^+$ the set of functions on~$J_0$ taking values in~$[0,+\infty]$.
For every function $\varphi\:[0,1]\to \R$ and every~$p$ in~$\R$, define \emph{the transfer operator}~$\sL_{\varphi,p}$ acting on a 
function~$h$ in~$L^+$ by 
\begin{equation}
(\sL_{\varphi,p}h)(y) \= \sum_{x\in F^{-1}(y)} \exp(S_{m(x)}\varphi(x)-m(x)p )h(x).
\end{equation}
The following proposition relates the transfer operator $\sL_{\varphi,p}$ and the two-variable pressure $\sPf(\varphi,p)$.
\begin{prop}
\label{l:operator}
Let $\varphi\:[0,1]\to \R$ be a continuous potential with bounded variation ergodic sums on $J_0$ for~$f$ with constant $C>0$. 
The following properties hold.
\begin{enumerate}
\item[1.] For every $p$ in $\R$ the two-variable pressure $\sPf(\varphi,p)$ exists and is in $\R\cup \{+\infty\}$.
\item[2.] For all $y$  in $J_0$,  $p$ in $\R$ and $\ell$ in $\N$, we have
\begin{equation}
\label{e:Z vs L}
\exp(-C) Z_\ell(\varphi, p) \le (\sL_{\varphi,p}^\ell \one) (y)  \le Z_\ell( \varphi, p),
\end{equation}
and thus,
\begin{equation}
\label{e:transfer pressure}
\sPf(\varphi,p) = \lim_{\ell \to +\infty} \frac{1}{\ell} \log (\sL_{\varphi,p}^\ell \one )(y),
\end{equation}
where the limit is independent of $y$ in $J_0$. 
\item[3.] For all $y$ in $J_0$,  $p$ in $\R$ and  $\ell$ in $\N$, we have
\begin{multline}
\label{eq:Z distortion}
\frac{1}{\ell} \log\left( \frac{(\sL_{\varphi,p}^\ell  \one )(y)}{ \exp(C)}\right) \le \frac{1}{\ell} \log\left( \frac{Z_\ell(\varphi, p)}{ \exp(C)} \right) \le \sPf(\varphi,p) \\ 
\le  \frac{1}{\ell} \log Z_\ell(\varphi, p) \le 
\frac{1}{\ell} \log \left( \exp(C) (\sL_{\varphi,p}^\ell  \one )(y)\right).
\end{multline}
In particular, $\sP(\varphi,p)<+\infty$ if and only if $Z_1(\varphi,p)<+\infty$.
\end{enumerate}
\end{prop}

\proof
First, observe that since $\varphi$ has bounded variation ergodic sums on $J_0$ for $f$, for every $p$ in $\R$, and all $k$ and $\ell$ in $\N$, we have 
\begin{equation}
\label{e:kl}
(\exp(- C) Z_{\ell}(\varphi, p))^k  \le Z_{k\ell}(\varphi, p) \le Z_{\ell}(\varphi, p)^k. 
\end{equation}
Then, for $Z_{1}(\varphi, p)$ finite, the sequence $(\log Z_{\ell}(\varphi, p))_{\ell \in \N}$ is finite and subadditive. Thus, by the Subadditive Lemma $\sP(\varphi,p)$ exists and belongs to $\{-\infty\}\cup \R$. But, by \eqref{e:kl}, we have that $\sP(\varphi,p)$ is in $\R$.
For $Z_{1}(\varphi, p)=+\infty$, by \eqref{e:kl} for every $\ell$ in $\N$ we have $Z_{\ell}(\varphi, p)=+\infty$. Then, $\sP(\varphi,p)$ exists, and it is equal to $+\infty$. This proves item 1.

For all $y$ in $J_0$,  $p$ in $\R$ and $\ell$ in $\N$ we have
\begin{equation}
\label{e: L interation}
(\sL_{\varphi,p}^\ell h )(y) =  \sum_{x\in F^{-\ell}(y)} \exp(S_{m(x)+\cdots+m(F^{\ell-1}(x))}(\varphi-p)(x))h(x).
\end{equation}
Together with the bounded variation ergodic sums on $J_0$ for $f$ of $\varphi$, this implies
\begin{equation}
\label{e: Z vs L left}
\exp(-C) Z_\ell(\varphi, p) \le (\sL_{\varphi,p}^\ell \one) (y).
\end{equation}
Since the other inequality is always true, \eqref{e:Z vs L} holds, which implies \eqref{e:transfer pressure} and proves item 2.

Finally, from \eqref{e:transfer pressure} and \eqref{e:kl} we get
\begin{equation}
\frac{1}{\ell} \log \frac{Z_\ell(\varphi, p)}{ \exp(C)} 
\le \sP(\varphi,p)
\le \frac{1}{\ell} \log Z_\ell(\varphi, p).
\end{equation}
Together with  \eqref{e:Z vs L}, this implies \eqref{eq:Z distortion}, finishing the proof of the lemma.
\endproof

\subsection{Bowen-Type Formula}
\label{sss:Bowen}
In this section, we state the Bowen-Type Formula relating the pressure of the original system to the two-variable pressure of the induced system for potentials whose ergodic sums have bounded variation.

\begin{generic}[Bowen-Type Formula]
\label{p:Bowen type formula}
Let~$\varphi\:[0,1]\to \R$ be a continuous potential with bounded variation ergodic sums on~$J_0$ for~$f$. For each~$p$ in~$\R$, we have
\begin{displaymath}
  \sPf(\varphi, p)
  \begin{cases}
    > 0
    & \text{if } p < \Pf(\varphi);
    \\
    \le 0
    & \text{if } p = \Pf(\varphi);
    \\
    < 0
    & \text{if } p > \Pf(\varphi),
  \end{cases}
\end{displaymath}
and 
\begin{equation}\label{e:Bowen formula}
\Pf(\varphi)
=
\inf \left\{ p  \in \R \colon \sPf(\varphi, p) \le 0 \right\}. 
\end{equation}
\end{generic}
\proof
Let $C>0$ be the constant of the bounded variation ergodic sums on $J_0$ for $f$ of $\varphi$.
 For every $y$ in $(x_1,1]$ put 
\begin{equation}
  L_ p(y)
   \=
  1 + \sum_{z \in L^{-1}(y)} \exp(S_{m(z)}\varphi(z) - m(z)p ).
     \end{equation}
  Notice that, if we put $M\=\exp(\|\varphi \| + p)$, then
  \begin{equation}
    \label{eq:32}
   1 + M^{-1}\sL_{\varphi,p}\one(y)  \le  L_{ p}(y) \le 1 + M \sL_{\varphi,p}\one(y).
  \end{equation}

 The proof of the proposition is divided into several parts.
  
  \partn{1} For every~$p_0$  such that~$\sPf(\varphi, p_0) > 0$,  
  we prove~$\Pf(\varphi) > p_0$.
  First we prove that there is~$p > p_0$ such that~$\sPf(\varphi, p) > 0$.
  If~$\sPf(\varphi, p_0)$ is finite, then the function~$p \mapsto \sPf(\varphi, p)$ is finite, continuous, and 
  strictly decreasing on~$[p_0, + \infty)$.
  It follows that there is~$p > p_0$ such that~$\sPf(\varphi, p) > 0$.
  If~$\sPf(\varphi, p_0) = + \infty$, then~$(\sL_{\varphi, p_0}\one)(1) = + \infty$ by Proposition~\ref{l:operator}\eqref{eq:Z distortion}.
  By the Monotone Convergence Theorem, for every decreasing sequence $(p_n)_{n\in \N_0}$ in $(p_0,+\infty)$ converging to $p_0$, we have that $(\sL_{\varphi, p_n}\one)(1)$ converges to $(\sL_{\varphi, p_0}\one)(1)$ as $n$ goes to $+\infty$. It follows that there is~$p > p_0$ such that~$(\sL_{\varphi, p}\one)(1)  > \exp(C)$, and therefore~$\sPf(\varphi, p)  > 0$ 
  by Proposition~\ref{l:operator}\eqref{eq:Z distortion}. This proves that in all of the cases, there is $p>p_0$ satisfying $\sP(\varphi,p)>0$.
  Since for each integer~$\ell \ge 1$, every point of~$F^{-\ell}(1)$ is a preimage of~$1$ by an iterate of~$\mpf$, we have 
  by Proposition~\ref{l:operator}\eqref{eq:Z distortion}
  \begin{equation}
  \begin{split}
 \sum_{m = 1}^{+ \infty}  \exp( -  &  mp)  \sum_{y \in \mpf^{-m}(1)} \exp( S_m\varphi(y))
   \\   
   \ge &  \sum_{\ell = 1}^{+ \infty} \sum_{y \in F^{-\ell}(1)} \exp(S_{m(\mpf^{\ell - 1}(y)) + \cdots + m(y)}(\varphi-p)(y)) 
    \\ 
    = & \sum_{\ell = 1}^{+ \infty} (\sL_{\varphi,p}^\ell \one )(1) \\
    = & + \infty.
    \end{split}
  \end{equation}
  Together with~\eqref{e:top pressure = tree pressure} in Proposition \ref{l:top pressure}(2), this implies~$\Pf(\varphi) \ge p > p_0$.

  \partn{2}
  For every~$p_0$ such that~$\sPf(\varphi, p_0) < 0$, we prove~$p_0 \ge \Pf(\varphi)$.
  Note that by Proposition~\ref{l:operator}\eqref{e:Z vs L}  and~\eqref{eq:32} there is~$\Delta_0>1$ such that for every~$y$ in~$(x_1,1]$ 
  we have~$L_{ p_0}(y) \le \Delta_0 L_{ p_0}(1)$. By Proposition~\ref{l:operator}\eqref{eq:Z distortion} , \eqref{eq:32} 
  and~$\sPf(\varphi, p_0) < 0$
  we get that $L_{ p_0}(1)$ is finite.
 Using $\sP(\varphi,p_0)<0$ and Proposition~\ref{l:operator}\eqref{eq:Z distortion} again, we obtain 
  \begin{equation}
  \begin{split}
    \sum_{m = 1}^{+ \infty}  \exp(- &mp_0) \sum_{y \in \mpf^{-m}(1)} \exp(S_m \varphi(y)) \\
    =
    &  L_{p_0}(1) + 
       \sum_{\ell = 1}^{+ \infty} \sum_{y\in F^{-\ell}(1)} L_{ p_0}(y) \exp(S_{m(\mpf^{\ell - 1}(y)) + \cdots + m(y)}(\varphi-p_0)(y))
      \\ 
      \le & \Delta_0 L_{ p_0}(1) \left(1 + \sum_{\ell = 1}^{+ \infty} \left(\sL_{\varphi,p_0}^\ell \one\right)(1) \right)
     \\ 
      < &
      + \infty.
      \end{split}
    \end{equation}
    In view of~\eqref{e:top pressure = tree pressure} in Proposition \ref{l:top pressure}(2), we obtain~$\Pf(\varphi) \le p_0$.

    \partn{3}
    Put
    \begin{displaymath}
      p_*
      \=
      \inf \left\{ p \in \R \colon \sPf(\varphi, p) \le 0 \right\}.
    \end{displaymath}
    We prove $p_*=P(\varphi)$ and~$\sPf(\varphi, \Pf(\varphi)) \le 0$. 
    For every~$p < p_*$ we have~$\sPf(\varphi, p) > 0$, and therefore~$\Pf(\varphi) > p$ by part~1.
 We conclude that~$\Pf(\varphi) \ge p_*$.
 To prove the reverse inequality, note that from the fact that~$p \mapsto \sPf(\varphi, p)$ is nonincreasing on~$\R$ and strictly decreasing on the set where it is finite, we have that this function is finite and strictly decreasing on~$(p_*, + \infty)$.
    It follows that for every~$p$ in~$(p_*, + \infty)$ we have~$ \sP(\varphi, p) < 0$, and therefore~$p \ge \Pf(\varphi)$ by item~2.
    We conclude~$p_* \ge \Pf(\varphi)$, and therefore~$p_* = \Pf(\varphi)$.
    Finally, observe that by item~1 we must have~$\sPf(\varphi, \Pf(\varphi)) \le 0$.
    
\partn{4} If $p_0<P(\varphi)$ then, by part~2, $\sP(\varphi,p_0)\ge 0$. But if $\sP(\varphi,p_0) = 0$, by \eqref{e:Bowen formula}, we get that $p_0\ge P(\varphi)$, which is a contradiction. Therefore, $\sP(\varphi,p_0) > 0$. 
If $p_0=P(\varphi)$ then, by part~1, $\sP(\varphi,p_0)\le 0$.
Finally, assume~$p_0>P(\varphi)$. By part~3, $\sPf(\varphi, \Pf(\varphi)) \le 0$.  Since the function $p \mapsto \sPf(\varphi, p)$ is strictly decreasing on the set where it is finite, we have that it is strictly decreasing on~$(P(\varphi), + \infty)$. Therefore, $\sP(\varphi,p_0)< 0$.
This finishes the proof of the proposition. 
\endproof

\subsection{Applications}
\label{ss:Applications}
We now state Proposition~\ref{p:omega transitions}, which is one of the main consequences of the Bowen-Type Formula. Its proof is at the end of this subsection; it also uses Lemma \ref{l:sarig class is loc holder}  and Proposition~\ref{l:operator}.
Recall that for every~$\gamma$ in~$(0,+\infty)$, we denote by~$\omega_\gamma$ the function from~$[0,1]$ to~$\R$ 
defined by~$\omega_\gamma(x)\= -x^{\gamma}$. 
Observe that~$\omega_\gamma$ is in~$C_\dag^{1,\gamma}(\R)$.
See also \cite[Proposition~3.3]{CorRiv25b}, for a proof of this result avoiding the inducing scheme.
\begin{prop}
\label{p:omega transitions}
Let~$\gamma$ be in~$(0,+\infty)$. If $\gamma>\alpha$, then, for every $\beta$ in $(0,+\infty)$, the inequality $P(\beta\omega_\gamma)>0$ holds. If $\gamma\le \alpha$, then, for every sufficiently large $\beta$ in $(0,+\infty)$, the inequality $P(\beta\omega_\gamma)=0$ holds.
\end{prop}

The following corollary is a direct consequence of Propositions~\ref{c:tpt} and~\ref{p:omega transitions}.

\begin{coro}\label{c:omega-phase-transition}
For every $\gamma$ in $(0,+\infty)$, the potential $\omega_\gamma$ undergoes a phase transition in temperature if and only if $\gamma\le \alpha$.
\end{coro}
The following corollary, together with~\cite[Theorem~A.1]{InoRiv26}, extends~\cite[Theorem~C]{Klo20} to \textsc{Manneville--Pomeau} maps.
\begin{coro}\label{c:no phase transition}
Let $\varphi\:[0,1]\to \R$ be a \textsc{Hölder} continuous potential. 
If there is $\delta$ in $(0,1]$ such that 
$\varphi \ge \varphi(0)$ on $ [0,\delta]$,  then $\varphi$ does not undergo a phase transition in temperature.
\end{coro}
\proof
Take~$\gamma > \alpha$, and
put~$m\= \inf_{x\in [\delta,1]} \varphi(x)-\varphi(0)$ and~$c\= -m\delta^{-\gamma}$. If~$m<0$, then~$\varphi(x)-\varphi(0)\ge c \omega_\gamma(x)$ on~$[\delta,1]$ and~$\varphi(x)-\varphi(0)\ge 0$ on~$[0,\delta]$. 
Hence,~$\varphi(x)-\varphi(0)\ge c \omega_\gamma(x)$ on~$[0,1]$ and  by Proposition~\ref{p:omega transitions},  for every $\beta$ in~$(0,+\infty)$, we have
\begin{equation}
P(\beta(\varphi-\varphi(0)))\ge P(\beta c\omega_\gamma) > \beta c\omega_\gamma(0)=0.
\end{equation}
Then, for every $\beta$ in~$(0,+\infty)$, we get~$P(\beta\varphi)> \beta\varphi(0)$.
By Proposition~\ref{c:tpt}, the potential $\varphi$ does not undergo a phase transition in temperature.
If $m\ge 0$ then  $\varphi(x)-\varphi(0)\ge \omega_\gamma(x)$ on $[0,1]$. The same argument shows that $\varphi$ has no phase transition in temperature, which concludes the proof of the result.
\endproof

We also use Proposition~\ref{p:omega transitions} to study phase transitions in temperature for several potentials already present in the literature. We start with Corollary~\ref{c:phase transitions}, another consequence of Proposition~\ref{p:omega transitions}. Then, we apply these results to the study of the phase transition in temperature of the geometric potential in Proposition~\ref{ex:geometric}. This is well known, but we will take the opportunity to provide a simple proof using the tools developed in this section. 

The following corollary gives some simple conditions on Hölder continuous potentials for having a phase transition in temperature.
\begin{coro}\label{c:phase transitions}
Let $\varphi\:[0,1]\to \R$ be a Hölder continuous potential such that~$\varphi-\varphi(0)$ is strictly negative on~$(0,1]$. The following hold.
\begin{enumerate}
\item[1.] If there are~$c>0$ and~$\delta>0$ such that for every $x$ in $[0,\delta]$,
\begin{equation}
\label{e:tame at 0}
\varphi(x)-\varphi(0) \le c\omega_{\alpha}(x),
\end{equation}
then~$\varphi$ undergoes a phase transition in temperature. 
\item[2.] For every $\gamma$ in $(0,\alpha]$, if $\varphi$ is a potential in $\Sspace$ with a nonzero leading coefficient
in $\Sspace$, then~$\varphi$ undergoes a phase transition in temperature. 
\end{enumerate}
\end{coro}
\proof
\partn{1} From the assumption that $\varphi-\varphi(0)$ is strictly negative on~$(0,1]$ we get that there is $c'>0$ such that for every $x$ in $[\delta,1]$ one has $\varphi(x)-\varphi(0)\le c'\omega_{\alpha}(x)$. Together with \eqref{e:tame at 0}, this implies that for $ \beta_0\=  \min\{c,c'\}$ we have 
$\varphi-\varphi(0)\le \beta_0 \omega_{\alpha}$ on $[0,1]$. But by Proposition~\ref{p:omega transitions}, for every sufficiently large~$\beta$, the equality $P(\beta\omega_\alpha)=0$ holds.  
Then,
\begin{equation}
\beta \beta_0^{-1} \varphi(0) \le P(\beta \beta_0^{-1} \varphi) \le  P(\beta \omega_{\alpha}) + \beta\beta_0^{-1}\varphi(0) = \beta\beta_0^{-1}\varphi(0),
\end{equation}
and $P(\beta \beta_0^{-1} \varphi) = \beta \beta_0^{-1} \varphi(0)$. Since $P(0\cdot \varphi) = \log 2 > 0$ by Proposition \ref{l:top pressure}(1), this implies that the function $\beta\mapsto P(\beta\varphi)$ fails to be real-analytic on $(0,+\infty)$.
That is, $\varphi$ undergoes a phase transition in temperature.

\partn{2} From Lemma \ref{l:potential form} and the fact that~$\varphi-\varphi(0)$ is strictly negative on~$(0,1]$ we get that the leading coefficient of 
$\varphi$ is nonpositive. Thus, by hypothesis, it should be negative. Again,  by Lemma \ref{l:potential form}, there are $c>0$ and $\delta>0$ such that 
\eqref{e:tame at 0} holds. From item 1, we conclude that $\varphi$ undergoes a phase transition in temperature.
\endproof

The geometric potential~$-\log Df$ is a typical example of a potential that undergoes a phase transition in temperature. It is known that this phase transition occurs at~$1$. As an application of the Bowen-Type Formula, we provide a simple proof of this fact here. An alternative proof, which avoids the inducing scheme, 
is given in~\cite[Proposition~4.5]{CorRiv25b}.

\begin{prop}[The geometric potential]
\label{ex:geometric}
The geometric potential~$-\log Df$ undergoes a phase transition in temperature at $1$.
\end{prop}
\proof
From Corollary~\ref{c:phase transitions}(2), we have that the geometric potential undergoes a phase transition in temperature. 
Now, we prove this phase transition in temperature occurs at $\beta=1$. 

First, observe that on~$\left(\alpha/(\alpha+1),+\infty\right)$ the 
function~$\beta\mapsto \sP(-\beta\log Df,0)$ is finite and thus strictly decreasing. By Lemma \ref{l:geometric-distortion}, there is a constant $C_1>1$ such that for every $n$ in $\N$ and for every~$J$ in~$\bigvee_{j=0}^{n-1} F^{-j}(\fD)$ we have 
\begin{equation}
 \inf_{x\in J} DF^n(x) \le \frac{|J_0|}{|J|}\le \sup_{x\in J} DF^n(x) \le C_1\inf_{x\in J} DF^n(x).
\end{equation}
Observe that for every $\ell$ in $\N$ we have $ \sum_{J\in \bigvee_{j=0}^{\ell-1} F^{-j}(\fD)} |J| = |J_0|$.
Since,
\begin{equation}
\sP(-\log Df,0) = \lim_{\ell \to +\infty} \frac{1}{\ell} \log \sum_{J\in \bigvee_{j=0}^{\ell-1} F^{-j}(\fD)} \sup_{x\in J} \exp(-\log DF^\ell(x))
\end{equation}
we get 
\begin{equation}
\begin{split}
0=  \lim_{\ell \to +\infty} \frac{1}{\ell} \log \left( \frac{C_1}{|J_0|} \sum_{J\in \bigvee_{j=0}^{\ell-1} F^{-j}(\fD)} |J| \right) \ge &  \sP(-\log Df,0) \\
\ge & \lim_{\ell \to +\infty} \frac{1}{\ell} \log\left( \frac{1}{|J_0|} \sum_{J\in \bigvee_{j=0}^{\ell-1} F^{-j}(\fD)} |J| \right) \ge 0.
\end{split}
\end{equation}
This proves~$ \sP(-\log Df,0)=0$. Together with the fact that~$\beta\mapsto \sP(-\beta\log Df,0)$ is strictly decreasing 
on~$\left(\alpha/(\alpha+1),+\infty\right)$, we get that for every~$\beta$ in $(\alpha/(\alpha+1),1)$, one has~$\sP(-\beta\log Df,0)>0$, and 
for every~$\beta\ge 1$, one has~$\sP(-\beta\log Df,0)\le 0$. Then,
by Proposition~\ref{p:Bowen type formula}, if~$\sP(-\beta\log Df,0)>0$ then~$P(-\beta\log Df)\neq 0$ and~$P(-\beta\log Df)\ge 0$, 
and thus,~$P(-\beta\log Df) > 0$.
Again, by Proposition~\ref{p:Bowen type formula}, if~$\sP(-\beta\log Df,0)\le 0$ then~$P(-\beta\log Df)\le 0$. 
Therefore, for every~$\beta$ in $(\alpha/(\alpha+1),1)$, one has~$P(-\beta\log Df)>0$, and 
for every~$\beta\ge 1$, one has~$P(-\beta\log Df)\le 0$. Since 
$P(-\beta\log Df)$ is always nonnegative we get that 
for every~$\beta\ge 1$, one has~$P(-\beta\log Df) = 0$.
By Proposition~\ref{c:tpt}, we conclude that $-\log Df$ undergoes a phase transition in temperature at 1.
\endproof

\proof[Proof of Proposition~\ref{p:omega transitions}]
For every  integer $n\ge 2$ we have  
\begin{equation}\label{e:Sn}
S_{n}\omega_\gamma(y_{n}) = - y_{n}^\gamma - \sum_{j=1}^{n-1} x_j^\gamma \text{ and } 
S_{n}\omega_\gamma(y_{n+1}) = - y_{n+1}^\gamma - \sum_{j=2}^{n} x_j^\gamma.
\end{equation}
Since $\omega_\gamma$ is decreasing, for every integer $n\ge 1$ and every $y$ in $(y_{n+1},y_n]$ we have 
\begin{equation}\label{e:omega decreciente}
S_n \omega_\gamma(y_{n}) \le S_n \omega_\gamma(y) \le S_n \omega_\gamma(y_{n+1}).
\end{equation}
By \eqref{eq:10} and \eqref{e:Sn} there are positive constants $C$ and $C'$ such that for all integer $n\ge 2$  and $y$ in $(y_{n+1},y_n]$ we have
\begin{equation}\label{e:birkof sum}
-1- C \sum_{j=1}^{n-1} \frac{1}{j^{\frac{\gamma}{\alpha}}} \le  S_{n}\omega_\gamma(y_{n}) \le S_n\omega_\gamma(y)\le 
S_{n}\omega_\gamma(y_{n+1}) \le - C' \sum_{j=2}^{n} \frac{1}{j^{\frac{\gamma}{\alpha}}}.
\end{equation}
 Observe  that for every $\beta$ in $(0,+\infty)$ we have
\begin{equation}\label{e:transfer}
(\sL_{\beta\omega_\gamma, 0} \one) (1)   = \sum_{x\in F^{-1}(1)} \exp(\beta S_{m(x)}\omega_\gamma(x))
 = \sum_{n=1}^{+\infty} \exp\left(\beta S_n\omega_\gamma(y_{n})\right).
\end{equation}
By \eqref{e:omega decreciente} we have 
\begin{equation}\label{e:L and Z 2}
Z_1(\beta\omega_\gamma,0) 
 \le 
\sum_{n=1}^{+\infty} \exp\left(\beta S_n\omega_\gamma(y_{n+1})\right).
\end{equation}

If $\gamma>\alpha$, then
\begin{equation}
\sum_{j=1}^{+\infty} \frac{1}{j^{\frac{\gamma}{\alpha}}} < +\infty.
\end{equation} 
By \eqref{e:birkof sum} and \eqref{e:transfer}, for every $\beta$ in $(0,+\infty)$ we have
\begin{equation}
(\sL_{\beta\omega_\gamma, 0} \one) (1) = +\infty.
\end{equation} 
By Proposition~\ref{l:operator}\eqref{eq:Z distortion}, we get~$\sP(\beta\omega_\gamma, 0)=+\infty$ and from the Bowen-Type Formula and Lemma \ref{l:sarig class is loc holder} we conclude $P(\beta\omega_\gamma)>0$. 

Now, we assume that $\gamma \le \alpha$ and prove that for $\beta$ sufficiently large $P(\beta\omega_\gamma)=0$. First, observe that it is enough to prove the result for $\gamma=\alpha$ because $\omega_\gamma\le \omega_\alpha$. Since the potential $\omega_\alpha$ is strictly negative on $(0,1]$, by \eqref{e:omega decreciente}, 
for every $n$ in $\N$ and every $y$ in $(y_{n+1},y_n]$ we have
\begin{equation}\label{e:menor que 1}
\exp(S_n\omega_\alpha(y)) \le \exp(S_n\omega_\alpha(y_{n+1})) < 1.
\end{equation}
Together with \eqref{e:birkof sum} and \eqref{e:L and Z 2} this  implies that for $\beta>0$ sufficiently large we have
\begin{equation}
Z_1(\beta\omega_\alpha,0)< 1.
\end{equation}
By Proposition~\ref{l:operator}\eqref{eq:Z distortion}, we get  $\sP(\beta\omega_\alpha,0)<0$, and thus, by the Bowen-Type Formula we deduce that $P(\beta\omega_\alpha)=0$.
\endproof

\section{Main Theorem and proof of Theorems~\ref{t:sarig},  \ref{t:sarig alpha} and \ref{t:robust sarig}}
\label{s:main}

In this section, we state and prove our principal technical result, the Main Theorem,  and prove Theorems~\ref{t:sarig},  \ref{t:sarig alpha} and \ref{t:robust sarig}. The proofs of the theorems are in \S\ref{ss:proofs}, and the proof of the Main Theorem is in \S\ref{s:proofs}. 

\begin{generic}[Main Theorem]
\label{t:q optimization}
Let $\gamma$ be in $(0,\alpha]$, put $\theta \= 1-\gamma/\alpha$, and let $D$ be the constant in Lemma~\ref{l:sarig class is loc holder} for 
$\gamma$.
Let $n_0$ in $\N$ such that
for every integer $n\ge n_0$ we have
\begin{equation}
\label{e:xn lower bound 0}
\frac{1}{2\alpha^\frac{\gamma}{\alpha}}n^{-\frac{\gamma}{\alpha}} < x_n^\gamma.
\end{equation}
Let $\varphi$ be in $\Sspace$, let $c$ be in $(-\infty,0)$ such that for every $x$ in $[0,x_{n_0}]$ we have
\begin{equation}
\label{e:sup in compact}
\varphi(x)-\varphi(0) < c x^{\gamma},
\end{equation}  
and let $m_0$ be the least integer satisfying
\begin{equation}
\label{e:m0 robust}
m_0 
\\
>
\begin{cases}
\left[2(n_0+1)^\theta + \frac{4\alpha^\frac{\gamma}{\alpha}\theta}{-c}\left(D|\varphi|_{1,\gamma}+2 n_0 \|\varphi\| \right)\right]^{\frac{1}{\theta}} , & \text{ if  } \gamma < \alpha;\\
(n_0+1)^2 \exp\left(\frac{4\alpha}{-c} \left(D|\varphi|_{1,\gamma}+2 n_0 \|\varphi\| \right) \right), & \text{ if  } \gamma = \alpha.
\end{cases}
\end{equation}
Then, the following are equivalent:
\begin{enumerate}
\item[1.] $\varphi$ undergoes a phase transition in temperature;
\item[2.] $\delta_0$ is the unique maximizing measure for $\varphi$;
\item[3.] \begin{equation}
\label{e:compact optimization}
\sup \left\{ \int \varphi \dd \nu \colon \nu \in \sM, \supp(\nu) \subseteq [x_{m_0}, 1] \right\} < \varphi(0).
\end{equation}
\end{enumerate}
Moreover, conditions \eqref{e:sup in compact}, \eqref{e:m0 robust}, and \eqref{e:compact optimization} are open in $\Sspace$. When the equivalent conditions 1,2 and 3 hold,  $\varphi$ undergoes a robust phase transition in temperature in $\Sspace$ and $\delta_0$ is, robustly in $\Sspace$, the unique maximizing measure for $\varphi$. 
\end{generic}

\begin{rema}
\label{r:n0}
By \eqref{eq:10} and Lemma \ref{l:potential form}, if $\varphi$ is a potential in $\Sspace$ with negative leading coefficient 
in~$\Sspace$, then there is~$n_0$ in~$\N$ satisfying Main Theorem \eqref{e:xn lower bound 0} and \eqref{e:sup in compact} with~$c$ equal to 2 times the 
leading coefficient of~$\varphi$.
\end{rema}

\subsection{Proof of Theorems~\ref{t:sarig}, \ref{t:sarig alpha}, and \ref{t:robust sarig}}
\label{ss:proofs}

\proof[Proof of Theorem~\ref{t:sarig}]
Under the assumptions of the theorem, we have that~$\varphi$ belongs to~$\Sspace$. 

\partn{1} Assume that  $\gamma\le \alpha$ and $c>0$.  Replacing $\varphi$ by $\varphi-\varphi(0)$ if necessary, assume  
$\varphi(0)=0$. 
By \eqref{eq:10}  and Lemma \ref{l:potential form}, there is~$n_0$ in~$\N$ such that for every integer $k > n_0$ and every $x$ in $(x_{k},x_{k-1}]$ we have
\begin{equation}
\label{e:c positive}
\varphi(x) \ge \frac{c}{2} (\alpha k)^{-\gamma/\alpha}.
\end{equation}
For every integer $n\ge 1$, let $p_n$ be the periodic point of $f$ of period $n$ in $(y_{n+1},y_{n}]$. We have that for every $j$ in $\{1,\ldots,n-1\}$,  $f^j(p_n)$ is in $(x_{n+1-j},x_{n-j}]$.  Together with \eqref{e:c positive} this implies 
\begin{equation}
S_n\varphi(p_n) \ge S_{n-1-n_0}\varphi(f(p_n)) - (n_0+1)\|\varphi\| \ge \left( \frac{c}{2} \sum_{k=n_0+1}^{n-1}\frac{1}{(\alpha k)^{\gamma/\alpha}} \right) 
- (n_0+1)\|\varphi\|.
\end{equation}
Since $\gamma\le \alpha$ and $c>0$, for $n$ sufficiently large, we get $S_n\varphi(p_n) > 0$. Thus, the invariant probability measure 
$\mu_{n}$ supported on the orbit of $p_n$ satisfies $\int \varphi \dd \mu_{n} > 0$, which implies that $\delta_0$ is not a maximizing measure for $\varphi$.

\partn{2} Assume that $\gamma\le \alpha, c < 0$ and $\delta_0$ is the unique maximizing measure for $\varphi$. By the Main Theorem and Remark \ref{r:n0}, the potential
$\varphi$ undergoes a phase transition in temperature.

\partn{3} Assume that $\varphi$ undergoes a phase transition in temperature. 
 \textcolor{black}{By Corollary \ref{c:ergdodic optimization}, we have that $\delta_0$ is the unique maximizing measure for~$\varphi$.} Suppose we had~$\gamma>\alpha$. By Lemma~\ref{l:potential form} there is~$c'$ in~$(0,+\infty)$ such that  for every~$x$ in~$[0,1]$  we have~$\varphi(x)-\varphi(0) \ge c' \omega_\gamma(x)$. 
By Proposition~\ref{p:omega transitions}, for every~$\beta>0$ we have 
\begin{equation}
P(\beta\varphi)-\beta\varphi(0) = P(\beta\varphi-\beta\varphi(0)) \ge  P(\beta c' \omega_\gamma)>0.
\end{equation}
Then, by Proposition~\ref{c:tpt}, the potential $\varphi$ does not undergo a phase transition in temperature, which is a contradiction. 
Therefore, $\gamma\le \alpha$ and $\delta_0$  the unique maximizing measure for~$\varphi$. By item 1, $c\le 0$, which concludes the proof of the theorem. 
\endproof

\proof[Proof of Theorem~\ref{t:sarig alpha}]
By Theorem~\ref{t:sarig}(3), the leading coefficient~$c$ of~$\varphi$ in~$C_\dag^{1,\alpha}(\R)$ is nonpositive. Suppose we had~$c = 0$. By Lemma \ref{l:potential form}
there is a continuous function~$h\:[0,1]\to \R$ such that $h(0)=0$ and for every $x$ in $[0,1]$ we have $\varphi(x)=\varphi(0)+h(x)x^\alpha$. Then, for every 
$\beta$ in $(0,+\infty)$ there is $n_1$ in $\N$  such that for every $x$ in $[0, x_{n_1}]$ we have 
\begin{equation}
-\frac{\alpha \beta^{-1}}{2} x^\alpha \le \varphi(x) - \varphi(0).
\end{equation}
By \eqref{eq:10}, taking $n_1$ larger, if necessary, for every integer $n\ge n_1$ we have  
\begin{equation}
 x_n^\alpha \le \frac{2}{\alpha n}.
\end{equation}
Then, there is a constant $K>0$ such that for every $n$ in $\N$ we get
\begin{equation}
\frac{K}{n^{\beta^{-1}}} \le \exp(S_n\varphi(x_n) - n\varphi(0)).
\end{equation}
Thus, 
\begin{equation}
(\sL_{\beta\varphi, \beta \varphi(0)} \one) (1) = +\infty.
\end{equation}
By the Bowen-Type Formula and Proposition~\ref{l:operator},   $P(\beta\varphi)>\beta\varphi(0)$. Thus, by Proposition~\ref{c:tpt}, the potential $\varphi$ does not undergo a phase transition in temperature, which is a contradiction and finishes the proof of the theorem. 
\endproof

\proof[Proof of Theorem~\ref{t:robust sarig}]
The implication~$1\Rightarrow 2$ follows from Theorem~\ref{t:sarig}(3). Now we prove~\mbox{$2\Rightarrow 3$}. 
Theorem~\ref{t:sarig}(1) shows that the leading coefficient~$c$ of~$\varphi$ in~$\Sspace$ is nonpositive. 
Let's demonstrate that $c$ is negative. Suppose we had~$c=0$. Then, for every~$\varepsilon>0$ the potential~$\varphi-\varepsilon \omega_\gamma$ has positive leading coefficient in~$\Sspace$. By Theorem~\ref{t:sarig}(1),~$\delta_0$ is not a maximizing measure for~$\varphi-\varepsilon \omega_\gamma$, which is a contradiction. Therefore,~$c<0$. Together with Theorem~\ref{t:sarig}(2), this implies that $\varphi$ undergoes a phase transition in temperature.

Finally, the implication~$3\Rightarrow 1$ follows from the Main Theorem and Remark~\ref{r:n0}, concluding the proof of the theorem.
\endproof

\subsection{Proof of the Main Theorem}
\label{s:proofs}

The proof of $1\Rightarrow 2$ follows from Proposition~\ref{c:tpt} and the Key Lemma, and the proof of $2\Rightarrow 3$ is a direct consequence of compactness. 
We prove $3\Rightarrow 1$.
Let $n_0$ be in $\N$  satisfying \eqref{e:xn lower bound 0}, 
let $\varphi$ be a potential in $\Sspace$,  and let $c$ be in $(-\infty,0)$  verifying  \eqref{e:sup in compact}. 
Let $m_0$ be as in the theorem's statement and notice that \eqref{e:m0 robust} implies $m_0>n_0$. 
We assume that \eqref{e:compact optimization} holds.

Put
  \begin{equation}
    \label{eq:153}
    \eta
    \=
    \sup \left\{ \int \varphi \dd \nu -\varphi(0) \colon \nu \in \sM, \supp(\nu) \subseteq [x_{m_0}, 1] \right\}.
  \end{equation}
  By \eqref{e:compact optimization}, we have
  \begin{equation}
    \label{eq:154}
    \eta
    <
    0.
  \end{equation}
Put
\begin{equation}
\label{e:constants}
C'\=\max\left\{\frac{c}{4\alpha^\frac{\gamma}{\alpha}\theta} ,\eta \right\} \text{ and } C''\=\max\left\{ \frac{c}{4\alpha}, \frac{\eta}{\log2} \right\}.
\end{equation}

We prove that for every  $\beta>0$ and every 
$\ell$ in $\N$, both sufficiently large, we have 
\begin{equation}
\label{e:Zell<1}
Z_{\ell}(\beta\varphi,\beta\varphi(0))< 1.
\end{equation} 
Thus, by Proposition~\ref{l:operator}\eqref{eq:Z distortion}, we get $\sP(\beta\varphi,\beta\varphi(0)) < 0$, and then, by the Bowen-Type Formula, we obtain $\beta\varphi(0)\ge P(\beta\varphi)$. Since the inequality $P(\beta\varphi) \ge \beta\varphi(0)$ is always true, Proposition~\ref{c:tpt} implies that $\varphi$ undergoes a phase transition in temperature.

To prove \eqref{e:Zell<1}, it is enough to show that for each $\ell\in \N$ sufficiently large every summand in 
$Z_{\ell}(\varphi,\varphi(0))$ is strictly less than 1, and that for each $\beta$ in $(0,+\infty)$ sufficiently large
\begin{equation}
\label{e:Zell finite}
Z_{\ell}(\beta\varphi,\beta\varphi(0))< +\infty.
\end{equation} 
From the definition of $Z_{\ell}(\varphi,\varphi(0))$ in \eqref{d:Zell}, the first statement is equivalent to prove that 
for $\ell\in \N$ sufficiently large and every $J$ in $\bigvee_{j=0}^{\ell-1}F^{-j}(\fD)$ we have
\begin{equation}
\label{e:energy<1}
\sup_{x\in J} \left(S_{m(x)+\cdots+m(F^{\ell-1}(x))} \varphi(x) -(m(x)+\cdots+m(F^{\ell-1}(x))\varphi(0)\right)  < 0.
\end{equation}
Before proving this inequality, we provide two preliminary estimates—particular instances of \eqref{e:energy<1}—which will be used to establish the general case.

 Observe that for every~$\ell$ in $\N$, every  $J$ in $\bigvee_{j=0}^{\ell-1}F^{-j}(\fD)$, and every~$x$ in $J$  the sum of the return times~$m(x)+\cdots+m(F^{\ell-1}(x))$ is constant. 
 
For the first preliminary estimate, let~$\ell'$ be in $\N$ and let~$J'$ be 
in~$\bigvee_{j=0}^{\ell'-1}F^{-j}(\fD)$. For~$z$ in~$J'$ put 
 $n' \=  m(z)+\cdots+m(F^{\ell'-1}(z)) $.
Observe that $f^{n'}|_{J'}=F^{\ell'}|_{J'}$ and denote by $p$ the unique periodic point of $f$ of period $n'$ in $J'$. Assume that the orbit segment $(f^j(z))_{j=0}^{n'}$ is included in $[x_{m_0},1]$. Thus, the orbit of $p$ is also included in $[x_{m_0},1]$. 
By \eqref{eq:153},  and since $\varphi$ has  bounded variation ergodic sums on $J_0$ for $f$, 
by Lemma \ref{l:sarig class is loc holder}, 
for every $z$ in $J'$ we have
\begin{equation}
\label{e:sombra}
S_{n'}\varphi(z)-n'\varphi(0) \le D|\varphi|_{1,\gamma} + S_{n'}\varphi(p)-n'\varphi(0) \le  D|\varphi|_{1,\gamma}+n'\eta.
\end{equation}

For the second preliminary estimate, let~$y$ be in~$J_0$ and assume that~$m(y) \ge m_0$. By \eqref{e:sup in compact},
\begin{equation}
\label{e:depth}
\begin{split}
S_{m(y)}\varphi(y)-m(y)\varphi(0) 
& =   \varphi(y)-\varphi(0)+  S_{m(y)-1}\varphi(f(y))-(m(y)-1)\varphi(0)
\\
& \le  2n_0\|\varphi\|   +c \sum_{j=n_0+1}^{m(y)} x_j^\gamma.
\end{split}
\end{equation}
Together with \eqref{e:xn lower bound 0} and \eqref{e:m0 robust}, for $\gamma<\alpha$, we get 
\begin{equation}
\label{e:depth 1}
\begin{split}
 S_{m(y)}\varphi(y)-m(y)\varphi(0) 
& \le 2n_0 \|\varphi\|  - \frac{c}{2\alpha^\frac{\gamma}{\alpha}\theta}   (n_0+1)^{\theta} +
\frac{c}{2\alpha^\frac{\gamma}{\alpha}\theta} (m(y)+1)^{\theta}  
\\
& \le  -D |\varphi|_{1,\gamma} +\frac{c}{4\alpha^\frac{\gamma}{\alpha}\theta} (m(y)+1)^{\theta},
\end{split}
\end{equation}
and for $\gamma=\alpha$,  we get
\begin{equation}
\label{e:depth 2}
\begin{split}
S_{m(y)}\varphi(y)-m(y)\varphi(0)
& \le 2n_0\|\varphi\|-\frac{c}{2\alpha}   \log(n_0+1) + 
\frac{c}{2\alpha} \log(m(y)+1)  
\\
& \le  -D|\varphi|_{1,\gamma} + \frac{c}{4\alpha} \log(m(y)+1).
\end{split}
\end{equation}

Now we prove \eqref{e:energy<1}.
\textcolor{black}{
Let~$\ell$ be in~$\N$, $J$ in~$\bigvee_{j=0}^{\ell-1}F^{-j}(\fD)$, and fix~$x$ in~$J$.
Moreover, put ${n \= m(x)+\cdots+m(F^{\ell-1}(x))}$
}
 and denote by $s$ the cardinality of the set 
\begin{equation}
\sD \= \{j\in \{0,\ldots,\ell-1\} \colon m(F^j(x)) \ge  m_0 \}.
\end{equation}
 If~$s=\ell$, from~\eqref{e:depth 1},~\eqref{e:depth 2} and the fact that the function~$r\mapsto r^\theta$ is subadditive on~$[0,+\infty)$ 
 when~$\gamma<\alpha$,  we get 
\begin{equation}
\label{e:only depth}
S_{n}\varphi(x)-n\varphi(0) \le 
\begin{cases}
-D|\varphi|_{1,\gamma}\ell + \frac{c}{4\alpha^\frac{\gamma}{\alpha}\theta} n^{\theta}, & \text{ if } \gamma< \alpha;\\
-D|\varphi|_{1,\gamma}\ell + \frac{c}{4\alpha}   \log n, & \text{ if } \gamma=\alpha.
\end{cases}
\end{equation}
Suppose $s<\ell$ and put
\begin{equation}
\sS \= \{0,\ldots,\ell-1\}\setminus \sD.
\end{equation}
Then $\sS$  is nonempty; it can be decomposed into blocks of consecutive numbers. 
Let $\ell'_1\cdots \ell'_k$ be one of these blocks for $k$ in $\N$. Put $n'\= m(F^{\ell'_1}(x))+\cdots+m(F^{\ell'_k}(x))$. By definition of
$\sS$, for every $i$ in $\{1,\ldots, k\}$ one has that $m(F^{\ell'_i}(x)) < m_0$. Then, the orbit segment $(f^j(F^{\ell'_1}(x)))_{j=0}^{n'-1}$ is included in $[x_{m_0},1]$. By \eqref{e:sombra} we get
\begin{equation}
\label{e:bloque sombra}
S_{n'}\varphi(F^{\ell_1}(x))-n'\varphi(0)\le D|\varphi|_{1,\gamma}+n'\eta.
\end{equation}
Denote by $t$ the number of maximal blocks of consecutive numbers in $\sS$. Observe that 
\begin{equation}\label{e:t and s}
t\le s+1.
\end{equation}
Denote by $\ell_1,\ldots,\ell_s$ the numbers in $\sD$ and 
for every maximal block of consecutive numbers $\ell'_1(j)\cdots \ell'_{k_j}(j)$ in $\sS$, for $j\in \{1,\ldots, t\}$, set 
\begin{equation}
n'_j\= m(F^{\ell'_1(j)}(x))+\cdots+m(F^{\ell'_{k_j}(j)}(x)).
\end{equation}
Put 
\begin{equation}
n''\= \sum_{j=1}^t n'_j
\end{equation}
and  notice that 
\begin{equation}
n=n''+ m(F^{\ell_1}(x))+\cdots + m(F^{\ell_s}(x)).
\end{equation}
For~$\gamma<\alpha$, using~\eqref{eq:154},\eqref{e:constants},\eqref{e:depth 1},\eqref{e:bloque sombra}  and~\eqref{e:t and s}, and the subadditivity of the 
function~$r\mapsto r^\theta$ on~$[0,+\infty)$  we get  that
\begin{equation}
\label{e:negative deviation 1}
\begin{split}
\quad S_{n}\varphi(x)&-n\varphi(0) \\
& \le -D|\varphi|_{1,\gamma} s  + \frac{c}{4\alpha^\frac{\gamma}{\alpha}\theta} \sum_{j=1}^s  (m(F^{\ell_j}(x))+1)^\theta  + D|\varphi|_{1,\gamma} t + \eta n''
\\
& \le D |\varphi|_{1,\gamma} + \frac{c}{4\alpha^\frac{\gamma}{\alpha}\theta}  
\left(\sum_{j=1}^s (m(F^{\ell_j}(x))+1)\right)^\theta +\eta (n'')^\theta
\\
& \le D |\varphi|_{1,\gamma} + C' n^\theta. 
\end{split}
\end{equation} 
For~$\gamma=\alpha$, from~\eqref{eq:154},\eqref{e:constants},\eqref{e:depth 2},\eqref{e:bloque sombra}  and~\eqref{e:t and s}, and using that for each $j$ in $\{1,\ldots,s\}$ one has that 
$m(F^{\ell_j}(x))>1$  we get that
\begin{equation}
\label{e:negative deviation 2}
\begin{split}
\quad S_{n}\varphi&(x)-n\varphi(0) \\
& \le -D|\varphi|_{1,\gamma}s + \frac{c}{4\alpha} \sum_{j=1}^s \log (m(F^{\ell_j}(x))+1)  + D|\varphi|_{1,\gamma} t + n''\eta
\\
& \le D|\varphi|_{1,\gamma} + \frac{c}{4\alpha}   \log \left( \prod_{j=1}^s (m(F^{\ell_j}(x))+1)\right) +\frac{\eta}{\log2}\log(n''+1)
\\
& \le D|\varphi|_{1,\gamma}+C''  \log \left( (n''+1) \prod_{j=1}^s (m(F^{\ell_j}(x))+1)\right) 
\\
& \le D |\varphi|_{1,\gamma}+C'' \log n. 
\end{split}
\end{equation} 
Observe that by \eqref{eq:154} and the fact that $c<0$, the constants $C'$ and $C''$ are strictly negative.
Since $n \ge \ell$, by  \eqref{e:only depth}, \eqref{e:negative deviation 1}, and \eqref{e:negative deviation 2} there is $\ell_0$ in $\N$ such that for every integer $\ell\ge \ell_0$  the inequality \eqref{e:energy<1} holds.  
On the other hand, by \eqref{e:depth 1} and \eqref{e:depth 2}, for $\beta>0$ sufficiently large $Z_1(\beta\varphi,\beta\varphi(0))$ is finite. Thus, for every $\ell$ in $\N$ we have that 
$Z_\ell(\beta\varphi,\beta\varphi(0))$ is also finite. Therefore, for $\beta>0$ and 
$\ell$ in $\N$, both sufficiently large, we get \eqref{e:Zell<1}, which prove that $\varphi$ undergoes a phase transition in temperature.

For the second part of the proposition, notice that by conditions \eqref{e:m0 robust} and \eqref{e:compact optimization} there is 
$\varepsilon$ in $(0,\min\{-c,-\eta\})$ such that
\begin{equation}
\label{e:m0}
m_0 
>
\begin{cases}
\left[2(n_0+1)^\theta +\frac{4\alpha^\frac{\gamma}{\alpha}\theta}{-(c+\varepsilon)}\left(D(|\varphi|_{1,\gamma}+\varepsilon)+2n_0(\|\varphi\|+\varepsilon) \right)\right]^{\frac{1}{\theta}} , & \text{ if  } \gamma < \alpha;\\
(n_0+1)^2\exp\left(\frac{4\alpha}{-(c+\varepsilon)} \left(D(|\varphi|_{1,\gamma}+\varepsilon)+2n_0(\|\varphi\|+\varepsilon)\right) \right), & \text{ if  } \gamma = \alpha.
\end{cases}
\end{equation}
Put $\widetilde{c}\= c+\varepsilon$, and observe that $\widetilde{c}<0$.  By \eqref{e:sup in compact}, if
$|\varphi-\tvarphi|_{1,\gamma}< \varepsilon$, then for every $x$ in $[0,x_{n_0}]$ we have
$\tvarphi(x)-\tvarphi(0) <  \widetilde{c} x^\gamma$,
which implies that \eqref{e:sup in compact} is an open condition in $\Sspace$.
Now, notice that for every $\tvarphi$ in $\Sspace$ such that $\|\tvarphi-\varphi\|_{1,\gamma}< \varepsilon$ we have 
$|\tvarphi|_{1,\gamma} < |\varphi|_{1,\gamma}+\varepsilon$ and $\|\tvarphi\| < \|\varphi\|+\varepsilon$.
Then, by \eqref{e:m0} we have \eqref{e:m0 robust} with $\varphi$ replaced by $\tvarphi$.
Thus, condition \eqref{e:m0 robust} is open in $\Sspace$.
Finally, since $\varepsilon< -\eta$ for every $\tvarphi$  satisfying $\|\varphi-\tvarphi\| < \varepsilon/2$,  we have   
\begin{equation}   
    \sup \left\{ \int \tvarphi \dd \nu -\tvarphi(0) \colon \nu \in \sM, \supp(\nu) \subseteq [x_{m_0}, 1] \right\} \le \eta + \varepsilon < 0.
  \end{equation}
  Showing that condition \eqref{e:compact optimization} is also open in $\Sspace$, which finishes the proof of the Main Theorem.

\section*{Acknowledgments}
We are grateful to the associate editor for the comments on Example~1.2 and on
the status of \cite[Proposition~1(2)]{Sar01a}. We also thank the referees for
their careful reading and suggestions, which helped us improve the exposition
and organization of the paper.


\end{document}